\documentclass[11pt,reqno]{amsart}

\usepackage[utf8]{inputenc} 

\usepackage[margin=1in]{geometry} 

\usepackage{graphicx} 
\usepackage{float} 

 \usepackage[parfill]{parskip} 
 
\usepackage{booktabs} 
\usepackage{array} 
\usepackage{paralist} 
\usepackage{verbatim} 
\usepackage{subfig} 
\usepackage{mathrsfs}
\usepackage{amssymb}
\usepackage{amsthm}
\usepackage{amsmath,amsfonts,amssymb,esint,hyperref}
\usepackage[noabbrev, capitalize]{cleveref}
\usepackage{graphics,color}
\usepackage{enumerate, enumitem}
\usepackage{mathtools,centernot}
\usepackage{cases}
\usepackage{amsrefs}
\usepackage{bbm}
\usepackage{xfrac}

\usepackage{bookmark}

\newtheorem{theorem}{Theorem}[section]
\newtheorem{lemma}[theorem]{Lemma}
\newtheorem{example}[theorem]{Example}
\newtheorem{corollary}[theorem]{Corollary}
\newtheorem{definition}[theorem]{Definition}
\newtheorem{proposition}[theorem]{Proposition}
\newtheorem{remark}[theorem]{Remark}

\numberwithin{equation}{section} 

\newcommand{\norm}[1]{\left\|#1\right\|}
\newcommand{\abs}[1]{\left|#1\right|}

\newcommand{\T}{\ensuremath{\mathbb{T}}}
\newcommand*{\R}{\ensuremath{\mathbb{R}}}

\newcommand{\eps}{\varepsilon}

\newcommand{\e}{\varepsilon}
\newcommand{\rmd}{{\rm d}}

\renewcommand{\MR}[1]{} 

\usepackage{color, graphicx}
\usepackage{mathrsfs, dsfont}

\usepackage[]{hyperref}
\hypersetup{
    colorlinks=true,       
    linkcolor=red,          
    citecolor=blue,        
    filecolor=red,      
    urlcolor=cyan           
}

\def\div{\mathop{\rm div}\nolimits}    
\def\dim{\mathop{\rm dim}\nolimits}
\def\spt{\mathop{\rm Spt}\nolimits} 
 
\def\Lip{\mathop{\rm Lip}\nolimits}

\newcommand{\be}{\begin{equation}}
\newcommand{\ee}{\end{equation}}

\title{On the Support of Anomalous Dissipation Measures}

\author{Luigi De Rosa}
\address{Department Mathematik Und Informatik, Universit\"at Basel, CH-4051 Basel, Switzerland}
\email{luigi.derosa@unibas.ch}

\author{Theodore D. Drivas}
\address{Department of Mathematics, Stony Brook University, Stony Brook, NY, 11794, USA}
\email{tdrivas@math.stonybrook.edu}

\author{Marco Inversi}
\address{Department Mathematik Und Informatik, Universit\"at Basel, CH-4051 Basel, Switzerland}
\email{marco.inversi@unibas.ch}

\date{\today}

\begin{document}

\begin{abstract}
By means of a unifying measure-theoretic approach, we establish lower bounds on the Hausdorff dimension of the space-time set which can support anomalous dissipation for weak solutions of fluid equations, both in the presence or absence of a physical boundary. Boundary dissipation, which can occur at both the time and the spatial boundary, is analyzed by suitably modifying the Duchon \& Robert interior distributional approach. One implication of our results is that any bounded Euler solution (compressible or incompressible) arising as a zero viscosity limit of Navier--Stokes solutions cannot have anomalous dissipation supported on a set of dimension smaller than that of the space.  This result is sharp, as demonstrated by entropy-producing shock solutions of compressible Euler \cites{Majda,DE18} and by recent constructions of dissipative incompressible Euler solutions \cites{brue2022,brue2022onsager}, as well as passive scalars \cites{drivas2022anomalous,crippa}. For $L^q_tL^r_x$ suitable Leray--Hopf solutions of the $d-$dimensional Navier--Stokes equation we prove a bound of the dissipation in terms of the Parabolic Hausdorff measure $\mathcal{P}^{s}$, which gives $s=d-2$ as soon as the solution lies in the Prodi--Serrin class. In the three-dimensional case, this matches with the Caffarelli--Kohn--Nirenberg partial regularity. 
\end{abstract}
\maketitle
\par
\medskip\noindent
{\textbf{MSC (2020):} 	35Q31 - 35L65 - 76D05 - 76F02.
\par
}

\section{Introduction}

Let $\Omega\subset \mathbb{R}^d$, $d\geq 2$, be a domain, possibly with smooth boundary $\partial \Omega$. If not otherwise specified, our spatial domain can be $\Omega=\R^d$ the whole space, $\Omega=\T^d$ the flat torus, or any general bounded domain $\Omega\subset \R^d$. Incompressible fluid motion is, to good approximation, described by the Navier--Stokes equations with no-slip boundary conditions for  a velocity field $u^\nu: \Omega \times (0,T)\to \mathbb{R}^d$
\begin{equation}\label{NS}
\left\{\begin{array}{l}
\partial_t u^\nu + (u^\nu \cdot \nabla)\, u^\nu = -\nabla p^\nu+ \nu \Delta u^\nu\qquad  \text{in}\  \Omega\times (0,T),\\
\div u^\nu=0\qquad \qquad\qquad\qquad\qquad\quad \,\,\,\, \ \ \  \  \text{in}\  \Omega\times (0,T),\\
u^\nu=0\qquad \qquad\qquad\qquad\qquad\qquad\quad \ \ \ \  \text{on}\  \partial \Omega\times (0,T).
\end{array}\right.
\end{equation}
The parameter $\nu>0$ represents the kinematic viscosity, which can be identified (upon non-dimensionalization with suitable characteristic scales $\mathsf{U}$ and $\mathsf{L}$) with the inverse Reynolds number $\mathsf{Re}^{-1}:= {\nu}/{\mathsf{UL}}$.  Turbulence intensifies as Reynolds number increases, with the limit $\mathsf{Re}\to \infty$  (equivalently zero viscosity) representing a regime of ``ideal" turbulence  in which the velocity field formally satisfies the \emph{inviscid} Euler equations, at least in a weak sense
\begin{equation}\label{E}
\left\{\begin{array}{l}
\partial_t u + (u \cdot \nabla)\, u = -\nabla p \qquad \quad \qquad \text{in}\  \Omega\times (0,T),\\
\div u=0\qquad \qquad\qquad\qquad\quad \,\,\, \ \ \ \   \text{in}\  \Omega\times (0,T),\\
u\cdot n=0 \qquad\,\,\,\qquad\qquad\qquad\quad \ \ \ \  \text{on}\  \partial \Omega\times (0,T),
\end{array}\right.
\end{equation}
where $n:\partial \Omega\rightarrow \R^d$ is the outward unit normal.
For strong solutions of the Euler equation, there is no mechanism for dissipation of energy.
Remarkably, observations from experiments  \cites{KRS84,PKW02} and simulations \cites{KRS98,KIYIU03,Farge} (for further discussion, see \cite{DE19} for the case without boundary and \cite{DN18} for the case with boundary) indicate that energy dissipation remains non-vanishing in ideal turbulence:
 \be\label{zerothLaw3d}
  \lim_{\nu\to 0}  \nu\langle |\nabla u^\nu|^2\rangle = \varepsilon>0,
 \ee
 where  $\langle \cdot \rangle$ is some relevant averaging procedure, space, time or ensemble. This phenomenon of  \emph{anomalous dissipation} \eqref{zerothLaw3d} is so central to our modern understanding of ideal turbulence that it has been termed the ``zeroth law". It is the principal postulate of the celebrated Kolmogorov 1941 (K41) theory \cite{K41} and reflects a turbulent cascade transferring energy from large to small scales which can effectively dissipate energy without the direct need of viscosity \cites{O49,Taylor17}. Under a space-time $L^3_{{\rm loc}}$ compactness hypotheses  on the family $\{u^\nu\}_{\nu>0}$ (implied by uniform regularity related to experimentally supported bounds on structure functions in the inertial range \cite{DN19}), the phenomenon \eqref{zerothLaw3d} indicates that ideal turbulence $u^\nu \to u$ is described by \emph{dissipative} weak Euler solutions $u$ \cite{DR00}, that is weak solutions for which the distribution
\be\label{limitmeasure}
\varepsilon[u] :=-\partial_t \left(\frac{|u|^2}{2}  \right) - \div  \left(\left(\frac{|u|^2}{2}+p\right)u\right),
\ee
which would be zero on classical solutions, is (locally) a non-trivial, non-negative space-time \emph{measure} $\varepsilon [u] \in \mathcal{M}_{{\rm loc}}(\Omega\times (0,T))$. In view of \eqref{zerothLaw3d}, $\langle\varepsilon[u] \rangle  =\varepsilon>0$.

A consequence of Kolmogorov's 1941 theory is that $\varepsilon[u]$ is supported on a full measure set of space-time, in keeping with prediction of a  ``monofractal" (having the same roughness everywhere) velocity field possessing, in some sense, exactly one-third of a derivative:
$$
|u(x+\ell z)-u(x)|\sim (\varepsilon \ell)^{1/3}, \quad |z|=1.
$$
However, as Landau famously remarked during Kolmogorov's Kazan lectures in 1942 (discussed subsequently in  section 34  of Landau \& Lifshitz \cite{LL}), the prediction of  a monofractal  velocity field is dubious precisely because of the space-time spottiness of the energy dissipation measure $\varepsilon[u]$. See also the account of Frisch \cite{Frisch91}.  That the velocity field is intermittent or multifractal (having fluctuating roughness in space and time) in high-Reynolds number turbulence is, by now, a well known experimental fact (see e.g. \cite{ISY20}). Moreover, simple singular solutions of model problems (such as shocks for the one-dimensional Burgers equation) clearly connect intermittency in the velocity field with anomalous dissipation supported on lower-dimensional sets.  A  ``multifractal model" based on energy dissipation was first proposed by Mandelbrot \cite{Mandelbrot} (see \cites{Meneveau88,Meneveau91} for experimental support), and Kolmogorov himself proposed to use a refined self-similarity of the dissipation to make multifractal predictions \cite{K62}. See Frisch \cite{Frisch95} and Eyink \cite{Enotes} for a survey and \cites{Eyink95,DRH23,CS1,CS2} for some mathematical results on multifractality.
A precise, rigorous connection between the fractal dimension of the support of anomalous dissipation  and intermittent scaling for incompressible Euler has recently been established \cites{Is17,DRI22}.

Consequently, an understanding of spatio-temporal variation of the energy dissipation measure \eqref{limitmeasure} is crucially important.  The first (and nearly only) study of this subject can be found in the pioneering series of works by Meneveau and Sreenivasan
\cites{Meneveau87,Meneveau88,Meneveau91}.  These papers suggest from experiments that, in the infinite Reynolds number limit, the anomalous energy dissipation measure at fixed time is concentrated on a subset of fractal dimension less than space dimension 3, about 2.87, and has volume zero \cite{Meneveau87}. Moreover, based on the data, it is supposed that this fractal dimension is (roughly) constant in time, making the inferred space-time support of the dissipation measure to be dimension 3.87.

In this paper, we are interested in establishing apriori \emph{lower bounds} for the dimension of the space-time support of anomalous dissipation based on integrability assumptions on the weak solutions.  We lay out a framework for proving lower bounds for weak solutions of general systems of conservation laws (including incompressible and compressible Euler, passive scalars, and Leray--Hopf solutions of Navier--Stokes) under integrability assumptions. Although some of our results can be stated under a single general theorem, for exposition purposes we divide the results in three categories: incompressible Euler, general conservation laws, incompressible Navier--Stokes.  Moreover,  for exposition purposes we limit ourselves to the case of positive dissipations but we remark that some of the results extends to the more general signed case.  Details are given in \cref{s:signed_case}.

\subsection{Incompressible Euler}
Before stating our main results, we clarify the subtle difference between interior and boundary dissipation. In the celebrated work of Duchon \& Robert \cite{DR00}, only interior dissipation is studied, i.e \eqref{limitmeasure} can be tested with $\varphi\in C^1_c(\Omega\times (0,T))$, giving
\begin{equation}\label{DR_measure_inter}
\int_0^T\int_\Omega  \left[ \frac{|u|^2}{2}\partial_t\varphi  + \left(\frac{|u|^2}{2}+p\right)u\cdot \nabla \varphi \right] \,\rmd x\, \rmd t= \int_0^T\int_\Omega \varphi \, \rmd \eps[u]
\end{equation}
as soon as $\eps[u]$ is a positive distribution, and thus it is locally a finite positive measure \cite{EG15}*{Theorem 1.39}. This is the case whenever $u$ arises as an $L^3_{{\rm loc}}(\Omega\times (0,T))$ limit of smooth, or more generally \emph{dissipative} in the sense of \cite{DR00}, solutions to incompressible Navier--Stokes. When $\Omega$ is either $\R^d$ or $\T^d$, standard Calderon--Zygmund estimates guarantee that $p\in L^\frac{3}{2}_{{\rm loc}}(\Omega\times (0,T))$, which then makes the term $pu$ well defined in \eqref{DR_measure_inter}. On a general bounded domain, the lack of a boundary condition on the pressure makes the issue delicate.

This approach does not trivially extend to analyze dissipation happening at the boundary (in time or space) since  it relies on standard regularization techniques. However, in view of formation of intense boundary layers at the physical boundary $\partial \Omega$  \cite{Farge}, and also of the recent constructions \cites{drivas2022anomalous,brue2022,brue2022onsager} for which all dissipation happens at the single instant $T$, we give the following definition.

\begin{definition}[Boundary extendable dissipation]    \label{d:extendable_diss}
Let $\Omega\subset \R^d$ be a bounded domain and let $u\in L^3(\Omega\times (0,T))$, $p\in L^\frac32 (\Omega\times (0,T))$ be a solution of the incompressible Euler system \eqref{E}. We say that $u$ has a positive boundary extendable dissipation if there exists a positive measure $\overline{\eps}[u]\in \mathcal{M}_{{\rm loc}}(\overline \Omega \times (0,T])$  such that 
\begin{equation}\label{DR_measure_boundary}
\int_0^T\int_\Omega  \left[ \frac{|u|^2}{2}\partial_t\varphi  + \left(\frac{|u|^2}{2}+p\right)u\cdot \nabla \varphi \right] \,\rmd x\, \rmd t=\int_0^T\int_{\overline \Omega} \varphi \,\rmd \overline \eps[u],\quad \forall \varphi\in C^1_c(\overline \Omega\times (0,T]).
\end{equation}
\end{definition}

In \cref{p:boundary dissipation for vanishing visc}, we show that if $\{(u^\nu,p^\nu)\}_{\nu>0}$ is a sequence of smooth Navier--Stokes solutions such that $u^\nu\rightarrow u$ in $L^3(\Omega\times (0,T))$ and $p^\nu \rightharpoonup p$ weakly in $L^\frac32(\Omega\times (0,T))$, then $(u,p)$ is a weak solution to the incompressible Euler system with boundary extendable dissipation. We will use the definition above for $(u,p)$ satisfying general integrability conditions, i.e. $u\in L^q(0,T;L^r(\Omega))$ and $p\in L^\frac{q}{2}(0,T;L^\frac{r}{2}(\Omega))$, with $q,r\geq 3$. These latter conditions can also be restored by vanishing viscosity sequences with an appropriate uniform (in viscosity) bound, see \cref{r:unif_bound}.  

We are interested in establishing lower bounds on the dimension of the space-time support of the measures $  \eps[u]$ and $ \overline \eps[u]$.  In fact, we can establish lower bounds on the \emph{concentration set}, which is in general much smaller than the support. Recall that a measure $\mu$ is said to be concentrated on a set $S$ if $\mu(A)=\mu(A\cap S)$ for all measurable $A$, while $\spt \mu$ is the complement of the union of all open sets $O$ such that $\int \varphi \,d\mu=0$, for all $\varphi\in C^0_c (O)$. Since any measure is concentrated on its support,  we can restrict to the case $S\subset \spt \mu$,  while in general there holds $ \spt \mu\subset \overline S$. In particular,  any lower bound on the dimension of $S$, translates into a dimension lower bound for $\spt \mu$.  As soon as the concentration set is dense, then $\spt \mu$ is the whole ambient space.  Since measures can concentrate on countable dense sets, then dimensional lower bounds on $S$ are clearly more useful than those on $\spt \mu$.

We denote by $\mathcal{H}^s_\alpha$, for $\alpha>0$, the generalised space-time Hausdorff measure (see \cref{s:tools} for precise definition), which imposes an $\alpha$ scaling in the time variable allowing for anisotropy of measurement in space and time. We remark that $\mathcal{H}^s_\alpha$ coincides with the usual Hausdorff measure $\mathcal{H}^s$ for $\alpha=1$ and with the parabolic Hausdorff measure $\mathcal{P}^s$ for $\alpha =2$. In the following result, we choose $\alpha$ in order to optimize the dimension bound. Indeed, the Euler equations do not impose any time-space scaling (as contrary, for instance, to Navier--Stokes, which imposes parabolic scaling). The choice of $\alpha$ in the theorem below coincides with the one which makes the corresponding $L^q_tL^r_x$ norm scaling invariant, see \cref{r:scaling_euler}.

\begin{theorem}[Incompressible Euler]\label{t:main_euler}
Let $\Omega\subset \R^d$, $d\geq 2$, be a domain and let $3\leq r,q\leq \infty$ be such that
$$ s = d \frac{r-2}{r} -\frac{2}{q-1}\frac{r+d}{r}\geq 0,\ \ \ \alpha = \frac{q}{q-1}\frac{r+d}{r}>0,$$
with $s=d$ and $\alpha=1$  whenever $r,q=\infty$. Let $u\in L^q_{{\rm loc}}(0,T;L^r_{{\rm loc}}(\Omega))$ and $p\in L^\frac{q}{2}_{{\rm loc}}(0,T;L^\frac{r}{2}_{{\rm loc}}(\Omega))$ be a weak solution of the incompressible Euler equations \eqref{E} having non-negative local anomalous dissipation measure $\varepsilon[u]\in \mathcal{M}_{{\rm loc}}(\Omega\times (0,T)) $. Then $\eps[u]\ll \mathcal{H}^{s}_{\alpha}$ and, if $\eps[u]$ is non-trivial and concentrated on a space-time set $S$,  we have 
\be\label{lowerbound}
\dim_{\mathcal{H}_\alpha}S\geq s.
\ee
Further, when $\Omega$ is bounded and if $(u,p)\in L^q(0,T; L^{r}(\Omega))\times L^{\frac{q}{2}}(0,T; L^{\frac{r}{2}}(\Omega))$  has a boundary extendable dissipation $\overline \eps[u]\in \mathcal{M}_{{\rm loc}}(\overline \Omega\times (0,T])$ in the sense of \cref{d:extendable_diss}, then the same conclusions hold for $\overline \eps[u]$.
\end{theorem}

A few remarks are in order. First, \cref{t:main_euler} is a direct consequence of the more general \cref{t:euler_general} (which is itself a consequence of \cref{p_euler_balls}) in which we also keep the time scaling parameter $\alpha$ free. We believe such a statement being helpful since there is in general no obvious relation between the measures $\mathcal{H}^s_\alpha$ if both $\alpha$ and $s$ vary. \cref{t:euler_general} should be compared with \cite{S05}*{Theorem 3.2} which we recall in \cref{t:div_measure} below. The latter shows that it is possible to deduce continuity properties (in terms of absolute continuity with respect to the Hausdorff measure) of the divergence of $L^r$ measure divergence vector fields. This approach turns out to be quite unifying, and directly applies to deduce continuity properties of dissipation and thus lower bounds on their dimension for much more general conservation laws (passive scalar transport, Burgers equation, multidimensional compressible and incompressible Euler equations, among others). Indeed, in all these cases the local energy balance directly falls in the setting of vector fields (space-time in our case) whose divergence is a positive bounded measure. Details and various examples are given in \cref{s:general_laws} below.  For further readings on divergence measure field, together with applications to systems of conservation laws, we refer to \cites{CT05,CTZ07,PT08,CTZ09,CT11} and references therein. However, since for fluid equations space and time usually scale very differently, our \cref{t:euler_general} (and thus \cref{t:main_euler} above) should be viewed as a specialized version of the approach in \cite{S05}. We emphasize that the two approaches coincide in the case of bounded solutions (i.e. $r,q=\infty$), while for general integrability conditions, our \cref{t:main_euler} gives sharper dimension bounds.  
\begin{remark}[Pressure hypothesis]\label{r:pressure_hyp}
We remark that in \cref{t:main_euler} the explicit hypothesis on the pressure $p\in L^\frac{q}{2}_{{\rm loc}}(0,T;L^\frac{r}{2}_{{\rm loc}}(\Omega))$ is redundant whenever $r<\infty$ by classical Calderon--Zygmund estimates, at least when $\Omega$ is the whole space $\R^d$ or the periodic box $\T^d$. However, when $r=\infty$ such an assumption is necessary. If $r=\infty$ and without further assumption on the pressure (but only on $u$), then one still has $p\in L^\frac{q}{2}_{{\rm loc}}(0,T;L^\frac{r}{2}_{{\rm loc}}(\Omega))$ for every $r<\infty$. Thus, our result would read as $\eps[u]\ll \mathcal{H}^{s-\gamma}_{\alpha}$, for all $\gamma>0$, where\footnote{To be precise, this conclusion follows by a minor modification of \cref{p_euler_balls}. Indeed, for any $r<+\infty$, it is enough to choose $\alpha = \frac{q}{q-1}$ in place of \eqref{alpha_opt} and modify \eqref{almost_final_est} accordingly. Otherwise \cref{t:main_euler} as stated would also have an arbitrarily small loss in the time exponent, i.e. $\alpha=\frac{q}{q-1}+\gamma'$ for all $\gamma'>0$.}
$$s= d- \frac{2}{q-1}, \quad \alpha = \frac{q}{q-1}.$$ 
Moreover, since $\gamma$ can be taken arbitrarily small, the dimension lower bound \eqref{lowerbound} in case of non-trivial dissipation stays the same (see \cref{general hausdorff dimension}). 

\end{remark}

Concerning the sharpness of \cref{t:main_euler}, the recent works \cites{brue2022,brue2022onsager,drivas2022anomalous} construct bounded sequences  $u^\nu\to u$ with a non-trivial dissipation measure supported at a single instant in time -- for this example, our theorem implies $\dim_{\mathcal{H}} \spt{\varepsilon[u]}= d$, showing that the lower bound  \eqref{lowerbound} cannot be improved.  This could, in fact, also be deduced by prior work of Leslie and Shvydkoy \cites{LS,LS2} which we discuss later.  Our result also implies that of Jeong and Yoneda \cite{JY21}*{Section 1.3.3}, who show that a $C^\alpha$ point singularity cannot support anomalous dissipation. Further explicit numerology in some physically relevant cases implied by our Theorem are discussed in \cref{s:numerology_euler}.
Note that the analog of \cref{t:main_euler} directly applies to the passive scalar example \cite{drivas2022anomalous}, but does not apply directly to the Euler solutions constructed in \cites{brue2022,brue2022onsager}, since those involve an external forcing. However, we remark that  (suitably integrable) external forces can be easily added to our analysis; details are sketched in \cref{s:ext_forc}.

Moreover, the recent convex integration constructions \cites{DK22, I_local22, GKN23} provide solutions with non-trivial and non-negative dissipation, on which our results directly apply. However, if one is only interested in getting lower bounds, then an adaptation of our proofs also applies to the signed case.  In \cref{s:signed_case} we will give all the necessary details.  We decided to state our main results in case of positive dissipation measures since in this case we can get stronger conclusions, such as the regularity of the dissipation in terms of the reference Hausdorff one, and (perhaps most importantly) global bounds on the local upper densities (see \cref{p_euler_balls}).

Finally, by the experiments by Meneveau and Sreenivasan \cites{Meneveau87,Meneveau88,Meneveau91}, assuming their flows remain bounded with increasing Reynolds number, the dimension of support of anomalous dissipation is 0.87 above our lower bound, indicating that -- unlike the simple shock singularities and mathematical constructions -- incompressible turbulence does not sit at the lower threshold of our theorem.  It would be interesting to obtain additional experimental and numerical measurements of the space-time support of the dissipation, perhaps leveraging a Lagrangian interpretation \cites{D19,C22}. 

In connection with our \cref{t:main_euler}, we mention the work \cite{LS2} by Leslie and Shvydkoy: with a fine oscillation/concentration analysis of the spatial dissipation measure at the first blowup time, they deduce regularity properties of the defect measure due to the possible failure of energy conservation. Thus, lower bounds on the dimension of the support of the spatial defect measure directly follow in the case of non-conservative solutions.  The main novelties of the present work are to establish lower bounds on the support of the full space-time dissipation measure (as opposed to dissipation on a hypothetical first singularity time slice) and, at the same time,  to provide a framework which is robust enough to include boundaries and general conservation laws.   As in \cite{LS2}, we emphasize that our results can be thought of as an energy conservation criterion, which directly follows from the absolute continuity property $\eps [u]\ll \mathcal{H}_\alpha^s$:  if $\eps[u]$ is concentrated on some $S$ with $\dim_{\mathcal{H}_\alpha} S <s$, then $\eps [u]\equiv 0$.  Moreover, our \cref{p_euler_balls} always achieves the stronger conclusion on the local asymptotic of the measure, from which the absolute continuity with respect to $\mathcal{H}^s_\alpha$ follows by \cref{AC with respect ot general hausdorff}. As we prove in \cref{p:counter_example}, having the right local asymptotic of a measure, is strictly stronger than the corresponding absolute continuity property.  The importance of the local behavior of dissipation measures in turbulent flows can be seen in its connection with the multifractal spectrum which naturally arises in these contexts  (see \cite{Falc}*{Chapter 17} as well as \cite{Frisch95}*{Chapter 8}). 

\subsection{General conservation laws setup}\label{s:general_laws}

Here we discuss how the measure theoretic setting of measure divergence vector fields from \cite{S05} applies to weak solutions of general systems of conservation laws, giving sharp dimension bounds in several aforementioned examples. We only discuss the local version (not taking the boundary into account), thus the argument applies to bounded and unbounded domains $\Omega$.

Let $u:\Omega\times(0,T)\to \mathbb{R}^N$ be a weak solution of
\begin{equation}\label{conslaw}
    \partial_t u + \div F[u] = 0
\end{equation}
for some flux $F:\mathbb{R}^N\to \mathbb{R}^{d \times N}$.  Suppose that smooth solutions of \eqref{conslaw} admit a conserved entropy $\eta:\mathbb{R}^N\to \mathbb{R}$ which, itself, satisfies a local conservation law
\begin{align}
\label{entropy_equality}
\partial_t \eta[u]+\div  Q [u] =0
\end{align}
for some corresponding entropy flux $Q:\mathbb{R}^N\to \mathbb{R}^d$.  We may regard the condition \eqref{entropy_equality} as a space-time divergence-free condition on a certain vector field $ \mathsf{V} :\Omega\times(0,T)\to \mathbb{R}^{N+1}$
\begin{equation*}
    \div_{t,x} \mathsf{V} =0, \qquad \mathsf{V}:=\Big(\eta[u],Q[u]\Big).
\end{equation*}

\begin{definition}[Entropy solutions]
Entropy solutions $u:\Omega\times(0,T)\to \mathbb{R}^N$ are weak solutions of \eqref{conslaw} such that
\begin{equation*}\label{entropyeqn}
\div_{t,x}\mathsf{V} =-\mathsf{D}
\end{equation*}
where $\mathsf{D}\in \mathcal{M}_{{\rm loc}}(\Omega\times (0,T))$ is a positive Borel measure, termed the anomalous dissipation measure.
\end{definition}
Such weak solutions may arise in various contexts as inviscid limits of non-ideal systems.  

\begin{example}[Incompressible Euler]
Weak solutions $u$ of incompressible Euler equations \eqref{E}
that arise as strong--$L^3$ vanishing viscosity limits of classical Navier--Stokes solutions $u^\nu\to u$ have \cite{DR00}
\begin{equation*}
 \mathsf{V} = \left(\frac{|u|^2}{2}, \left(\frac{|u|^2}{2}+ p\right)u\right), \qquad \mathsf{D}:=\lim_{\nu\to 0} \nu |\nabla u^\nu|^2.
\end{equation*}
\end{example}

\begin{example}[Incompressible Navier--Stokes]
Suitable Leray--Hopf weak solutions $u^\nu$ of the incompressible Navier--Stokes equations \eqref{NS} have (see \eqref{nsebal} below)
\begin{equation*}
 \mathsf{V} = \left(\frac{|u^\nu|^2}{2}, \left(\frac{|u^\nu|^2}{2}+ p^\nu\right)u^\nu- \nu \nabla \frac{|u^\nu|^2}{2}\right), \qquad \mathsf{D}:=D[u^\nu]+\nu |\nabla u^\nu|^2,
\end{equation*}
with $D[u^\nu]$ a positive distribution, thus locally a positive Radon measure.
\end{example}

\begin{example}[Transported Scalars]
Weak solutions $\theta$ of  scalar transport equation by a bounded, divergence-free vector field $u$
\begin{equation*}
\partial_t \theta + \div(u \theta)=0
\end{equation*}
that arise as zero-diffusion  limits have
\begin{equation*}
 \mathsf{V} = \left(\frac{\theta^2}{2}, u\frac{\theta^2}{2} \right), \qquad \mathsf{D}:=\lim_{\kappa\to 0}\kappa |\nabla \theta^\kappa|^2.
 \end{equation*}
\end{example}

\begin{example}[Compressible Euler]
 Weak solutions $(\rho,u,E)$ of full compressible Euler system\footnote{In this context the pressure $p$ is prescribed by a thermodynamic equation of state $p=p(u,\rho)$.}
\begin{equation*}
\left\{\begin{array}{l}
 \partial_t \rho + \div ( \rho u)= 0\\
 \partial_t (\rho u) + \div(\rho u \otimes u + p I) = 0 \\
\partial_t E + \div ((p+E) u) = 0,
\end{array}\right.
\end{equation*}
that arise as strong--$L^r $ vanishing viscosity limits of compressible Navier--Stokes \cite{DE18}*{Equation (16)} have
\begin{equation}\label{en_bal_compressible}
 \mathsf{V} = \Big(s, us\Big), \qquad \mathsf{D}:=\lim_{\eta,\zeta,\kappa\to 0}\left[ 2\eta \frac{|S[u]|^2}{T} + \zeta \frac{|\div u|^2}{T} +\kappa \frac{|\nabla T|^2}{T^2}\right]
\end{equation}
where  $s=s(u,\rho)$ is the physical entropy,  $T=T(u,\rho)$ is the temperature of the fluid (both of them defined by an appropriate equation of state) and 
$$
S[u]:= \frac{1}{2}\left(\nabla u+ (\nabla u)^t - \frac{2}{d} (\div u) I\right)
$$
is the  strain tensor.  In \eqref{en_bal_compressible} $\eta$ is the shear viscosity, $\zeta$ is the bulk viscosity and $\kappa$ is the thermal conductivity of the fluid.
\end{example}
For the sake of simplicity, we state only the result about interior dissipation, under a single space-time $L_{\rm loc}^r(\Omega\times (0,T))$ integrability assumption. Thus, we study the isotropic Hausdorff measure $\mathcal{H}_\alpha^s$ for $\alpha=1$. 
As already discussed before, we stress that the two approaches (space-time anisotropic vs isotropic) give the same results in case of bounded solutions. 

\begin{theorem}[General conservation laws]\label{t:main_general_laws}
Let $d\geq 1$, $\frac{d+1}{d}\leq r\leq \infty$ and $u:\Omega\times (0,T)\rightarrow \R^d$ be an entropy solution to \eqref{conslaw}, with $\mathsf{D}\in \mathcal{M}_{{\rm loc}}(\Omega\times (0,T))$ the corresponding anomalous dissipation measure. Assume that the entropy and the entropy flux enjoy $\eta[u],Q[u]\in L^r_{{\rm loc}}(\Omega\times (0,T))$. Set  $s=d+1-\frac{r}{r-1}$, with the usual convention that $s=d$ for $r=\infty$. Then, if $r<\infty$, we have that  
$$
\mathcal{H}^s(B)<\infty\,  \Rightarrow \,\mathsf{D}(B)=0, \quad \forall B\subset\joinrel\subset \Omega\times (0,T) \text{ Borel},
$$
while for $r=\infty$ we have $\mathsf{D}\ll \mathcal{H}^d$. In particular, if $\mathsf{D}$ is non-trivial and concentrated on a space-time set $S$, we have the dimensional lower bound 
$$
\dim_{\mathcal{H}}S\geq s.
$$
\end{theorem}

For bounded solutions, \cref{t:main_general_laws} shows that non-trivial dissipation has to happen on a set of Hausdorff dimension at least $d$, and the recent examples \cites{drivas2022anomalous,crippa} show that our lower bound is realized, thus proving optimality of the exponent. 
For compressible systems (Burgers equation, multidimensional Euler), \emph{entropy producing} shock solutions exhibit discontinuities on co-dimension one surfaces in space that evolve smoothly in time and have entropy production supported on a space-time set of dimension $d$   \cites{Majda,BDSV}. Shock solutions are bounded weak solutions which lie in the space $L^\infty(0,T; B_{r,\infty}^{1/r}(\Omega))$ for all $r\geq 1$, thus having a well-defined entropy measure \cite{DE18}.

\subsection{Suitable Leray--Hopf solutions of Navier--Stokes}
By means of the same techniques, we study the dissipation for the incompressible Navier--Stokes equations. Since the argument is completely local, it directly applies to $\Omega$, the full space $\R^d$, the periodic box $\T^d$ and any general bounded domain if no-slip boundary conditions are imposed.

To define \emph{dissipation} in the viscous context, let $u^\nu \in L^\infty(0,T;L^2(\Omega)) \cap L^2(0,T;H^1(\Omega))$ be a   Leray--Hopf weak solution to the Navier--Stokes system \eqref{NS} in $\Omega\times (0,T)$. Recall that as soon as $u^\nu \in L^3_{{\rm loc}}(\Omega\times (0,T))$, which holds without further assumption if $d=3$, the following energy identity holds in the sense of distributions
\begin{equation}\label{nsebal}
    \partial_t \frac{|u^\nu|^2}{2}+\div \left(\left(\frac{|u^\nu|^2}{2}+p^\nu \right) u^\nu \right) - \nu\Delta  \frac{\abs{u^\nu}^2}{2}  =-D[u^\nu] - \nu\abs{\nabla u^\nu}^2.
\end{equation}
Here, $D[u^\nu]$ is the  dissipation distribution associated to $u^\nu$, see Duchon \& Robert \cite{DR00}. The distribution $D[u^\nu]$ measures the failure of the local energy balance for viscous fluids due to possible singularities, indeed $D[u^\nu]\equiv 0$ whenever $u$ is smooth. We recall that for \emph{dissipative} solutions in the sense of \cite{DR00} (or analogously \emph{suitable} solutions as defined in \cite{CKN}) $D[u^\nu]$ is a positive distribution, and thus it is a positive locally finite measure. Remarkably, at least in the most physically relevant three dimensional case, the existence of dissipative weak solutions follows by the very same original argument by Leray \cite{L34}.

Because of the parabolic scaling preferred by the Laplacian, we work with the parabolic Hausdorff measure $\mathcal{P}^s$, together with the corresponding parabolic Hausdorff dimension $\dim_{\mathcal{P}}$ (see \cref{s:tools} for precise definition). 
 We prove the following result. 
\begin{theorem}[Leray--Hopf solutions in $L^q_tL^r_x$]\label{t:leray_hopf_intro}
Let $\Omega\subset \R^d$, $d\geq 2$, be a domain, possibly with boundary. Let $u^\nu \in L^\infty(0,T;L^2(\Omega)) \cap L^2(0,T;H^1(\Omega))$ be a Leray--Hopf weak solution to the incompressible Navier--Stokes equations \eqref{NS} such that $D[u^\nu]\in \mathcal{M}_{{\rm loc}}(\Omega\times (0,T))$. Assume $u^\nu\in L_{{\rm loc}}^q(0,T;L^r_{{\rm loc}}(\Omega))$ and $p^\nu\in L^\frac{q}{2}_{{\rm loc}}(0,T;L_{{\rm loc}}^\frac{r}{2}(\Omega))$,  for some $q,r\geq 3$ such that $\frac{2}{q}+\frac{d}{r}\geq1$ and 
$$
s=d+1-3\left(\frac{d}{r}+\frac{2}{q}\right)\geq 0.
$$
Then, $D[u^\nu]\ll \mathcal{P}^s$ and, if the dissipation is non-trivial and concentrated on a space-time set $S$, we have the dimensional lower bound
$$
\dim_{\mathcal{P}} S \geq s.
$$
\end{theorem}
 \cref{r:pressure_hyp} on the pressure hypothesis applies also here. Differently from Euler, for Leray--Hopf weak solutions of the three-dimensional Navier--Stokes system, the integrability on the pressure on general bounded domains follows by that of the velocity. Details are given in \cref{s:pressure_leray}.

The theorem above is a particular case of the more general \cref{t:leray-hopf_general} (which follows by \cref{p:leray-hopf_balls}) stated for any Hausdorff measure $\mathcal{H}^s_\alpha$. As a byproduct, we also infer the Morrey-type estimate \eqref{morrey-type_grad} on $|\nabla u^\nu|^2$, which, to the best of our knowledge, is new. We refer to \cref{r:scaling_NS} for further comments on the possible choices of $\alpha$ in the Navier--Stokes setting. The assumption $\frac{2}{q}+\frac{d}{r}\geq 1$ is natural in the sense that the well known Prodi--Serrin criterion \cites{P59,S62,S83} implies smoothness of $u$, and thus $D[u^\nu]\equiv 0$, whenever $\frac{2}{q}+\frac{d}{r}=1$ with $r>d$. When $d=3$, see also \cite{ISS03} for the limiting case $u^\nu\in L^\infty(0,T;L^3(\Omega))$. Moreover, in the three dimensional setting, \cref{t:leray_hopf_intro} is sharp in the following sense. If $u^\nu$ belongs to the Prodi--Serrin regularity class $\frac{2}{q}+\frac{3}{r}=1$, then $s=1$. By the Caffarelli--Kohn--Nirenberg partial regularity \cite{CKN}, we have that $\mathcal{P}^1(S_u)=0$ (where $S_u$ denotes the singular set, if any). Thus, we infer that $D[u^\nu]\equiv 0$. Indeed, since $\spt D[u^\nu]\subset S_u$, the property $D[u^\nu]\ll \mathcal{P}^1$ implies that $D[u^\nu]$ is necessarily trivial,  matching the fact that any such solution is smooth.  We refer to the recent works \cites{Wu21,Wu22} for partial regularity results of Navier--Stokes in higher space dimension.

Let us recall that global energy conservation (i.e. integrating \eqref{nsebal} in space) has been known to hold since the celebrated work of Lions \cite{L60} for all $L^4(\Omega\times (0,T))$ Leray--Hopf solutions in three dimensions.  Later on this has been generalized in \cite{S74} to $u\in L^q((0,T);L^r(\Omega))$ if $\frac{2}{q}+\frac{2}{r}\leq 1$, for some $r\geq 4$.
More recent related results in studying conditions implying the total energy balance can be found in \cites{LS,LS2}. In particular, in \cites{LS} Leslie and Shvydkoy deduce that global energy equality must hold if the singular set $S_u$ has small enough parabolic Hausdorff dimension. Indeed, the energy conservation property is also a consequence of the fact that $D[u^\nu]\ll \mathcal{P}^s$, as soon as $\dim_{\mathcal{P}}\big( \spt D[u^\nu]\big)<s$. The results from \cite{LS} do not directly follow from our \cref{t:leray_hopf_intro} since the two numerologies do not match in all the ranges of $r,q$.  Even if closely related to \cites{LS,L60,S74}, our \cref{t:leray_hopf_intro} aims to analyse the regularity properties of $D[u^\nu]$, yielding to non-trivial consequences in ranges in which the dissipation is not necessarily zero.  Moreover,  with respect to \cite{LS},  \cref{t:leray-hopf_general} deals with the more general reference measure $\mathcal{H}^s_\alpha$ instead of fixing the parabolic scaling.  As well as we did for the Euler equations, also here (see \cref{p:leray-hopf_balls}) we provide explicit asymptotic of the measure on space-time  cylinders,  yielding the corresponding uniform bounds on  fractional densities of the dissipation with respect to $\mathcal{H}^s_\alpha$.

\section{Measure theoretic tools}\label{s:tools}

Here we list the main tools used in the manuscript and outline the strategy of proof. 

\subsection{Anisotropic Hausdorff measure}

Given $\e, \delta>0$, we denote by
$$\mathcal{C}^\alpha_{\delta} (x,t) = B_\delta(x) \times (t-\delta^\alpha, t+\delta^\alpha) $$
the space-time cylinder centred at $(x,t)$.

\begin{definition}[Anisotropic Hausdorff measure]\label{d:hausdorff_meas}
Fix $\alpha \geq 0$. We denote by 
$$\mathcal{U}_\alpha = \{ \mathcal{C}^\alpha_{\delta}(x,t) \ | \ x \in \R^d, \  t \in \R, \ \delta >0 \}. $$
Fix $s \geq 0$ and let $\Psi^s : \mathcal{U}_\alpha \to [0, + \infty]$ be the gauge function defined by 
$$\Psi^s (\mathcal{C}^\alpha_{\delta}(x,t)) = \delta^s. $$
We denote by $\mathcal{H}_\alpha^s$ the measure on $\R^d \times \R$ associated to the family $\mathcal{U}_\alpha$ and the gauge function $\Psi^s$ via the Caratheodory construction (see for instance \cite{MAT}).
\end{definition}

We remark that $\mathcal{H}_\alpha^s $ is an anisotropic $s-$dimensional Hausdorff measure. Indeed, $\mathcal{H}_1^s$ is the standard $s-$dimensional Hausdorff measure $\mathcal{H}^s$ and $\mathcal{H}_2^s$ is the parabolic $s-$dimensional Hausdorff measure, usually denoted by $\mathcal{P}^s$. More generally, $\alpha$ can be thought as a time (in)stability parameter: for a fixed $s$, larger $\alpha$ is stronger is the notion of the measure, i.e. 
\begin{equation}\label{hausdorf_meas_instabil}
    \mathcal{H}^s_{\alpha}(A)\leq \mathcal{H}^s_{\alpha'}(A),\quad \text{ whenever } \alpha<\alpha',
\end{equation}
for every Borel set $A\subset \R^d\times \R$.
\begin{definition}[Anisotropic Hausdorff dimension] \label{general hausdorff dimension}
Let $\alpha \geq 0$ and let $A \subset \R^d \times \R$ be a Borel set. We define the $\alpha-$Hausdorff dimension of $A$ by
$$\dim_{\mathcal{H}_\alpha} (A) = \inf \{s \geq 0 \ | \ \mathcal{H}_\alpha^s(A) = 0 \}. $$ 
\end{definition}

By the above discussion, the notion of $\alpha-$Hausdorff dimension agrees with that of Hausdorff dimension for $\alpha=1$ and parabolic dimension for $\alpha=2$. Moreover, by \eqref{hausdorf_meas_instabil}, if $\alpha \leq \alpha'$ then $\dim_{\mathcal{H}_\alpha}(A)\leq \dim_{\mathcal{H}_{\alpha'}}(A)$.

We say that $\omega: [0, +\infty) \to [0, +\infty)$ is a modulus of continuity if it is non decreasing and $\omega(\delta) \to 0$ as $\delta \to 0^+$. We state and prove the following basic lemma. 

\begin{lemma}[Absolute continuity vs cylinders asymptotic] \label{AC with respect ot general hausdorff}
Fix $\alpha,s \geq 0$ and let $\mu$ be a positive measure on $\R^d \times \R$. Let $\delta_0>0$ and assume that for any $(x,t) \in \spt{\mu}$ and for any $\delta \in (0, \delta_0)$ it holds that 
\begin{equation}
    \mu\left(\mathcal{C}^\alpha_{\delta}(x,t) \right) \leq \omega(\delta) \delta^s, \label{smallness}
\end{equation}
where $\omega: [0, +\infty) \to [0, +\infty)$ is a bounded non decreasing function. Then, $\mu$ is absolutely continuous with respect to $\mathcal{H}_\alpha^s$. If we assume in addition that $\omega$ is a modulus of continuity, then $\mu(A) = 0$ for any Borel set $A\subset \R^d \times \R$ such that $\mathcal{H}_\alpha^s(A) < +\infty$. 
\end{lemma}

\begin{proof}
Let $A \subset \spt{\mu}$ be a null set with respect to $\mathcal{H}_\alpha^s$. Then, for any $\e>0$ there exists a covering $\{\mathcal{C}^\alpha_{r_i}(x_i, t_i)\}_{i\in I}$ of $A$ with $r_i \leq \delta_0$ for any $i \in I$ and such that 
\begin{equation}
    \sum_{i \in I} r_i^s \leq \e. \nonumber
\end{equation}
In particular, the family $\{\mathcal{C}^\alpha_{r_i}(x_i, t_i)\}_{i\in I}$ is countable. Then, we have that 
\begin{equation}
    \mu(A) \leq \sum_{i \in I} \mu(\mathcal{C}^\alpha_{r_i}(x_i, t_i)) \leq \omega(\delta_0) \sum_{i \in I} r_i^s \leq \omega(\delta_0) \e . \nonumber
\end{equation}
Since $\e>0$ is arbitrary, we infer that $\mu(A) = 0$. 

Assume in addition that $\omega$ is a modulus of continuity. Given a Borel set $A \subset \spt{\mu}$ such that $\mathcal{H}_\alpha^s(A)$ is finite and $\delta \in (0, \delta_0)$, we find a covering $\{\mathcal{C}^\alpha_{r_i}(x_i, t_i)\}_{i\in I}$ of $A$ such that $r_i \leq \delta$ for any $i \in I$ and 
$$\sum_{i\in I} r_i^s \leq \mathcal{H}_\alpha^s(A)+1< +\infty. $$
As before, we deduce that 
\begin{equation}
    \mu(A) \leq \sum_{i \in I} \mu(\mathcal{C}^\alpha_{r_i}(x_i, t_i)) \leq \omega(\delta) \sum_{i \in I} r_i^s \leq \omega(\delta) \left(\mathcal{H}_\alpha^s(A)+1\right). \nonumber
\end{equation} 
Letting $\delta \to 0^+$, we infer that $\mu(A) = 0$. 
\end{proof}
The previous lemma shows that a quantitative asymptotic behavior on cylinders implies absolute continuity of the measure with respect to a suitable anisotropic Hausdorff measure, according to the scaling between time and space of the cylinders. Indeed, this is the general strategy to prove all the results contained in this paper.  From \cref{AC with respect ot general hausdorff} we obtain the following result. 

\begin{corollary}[Support lower bound]\label{c:lowerbound_support}
Fix $\alpha\geq 0,s \geq 0$ and let $\mu$ be a positive measure on $\R^d \times \R$ that satisfies \eqref{smallness}. Assume that $\mu$ is non trivial and concentrated on a set $S$. Then, we have the dimensional lower bound
$$\dim_{\mathcal{H}_\alpha}S \geq s. $$ 
\end{corollary}

\begin{proof}
Since $\mu$ is non-trivial and concentrated on  $S$, it holds $\mu(S)>0$. Thus, by  \cref{AC with respect ot general hausdorff} we infer that $\mathcal{H}_\alpha^s(S) >0$. Hence, by \cref{general hausdorff dimension}, we deduce that $\dim_{\mathcal{H}_\alpha}S\geq s$.
\end{proof}

To contextualize our results, we note that -- for the standard Hausdorff measure -- the statement of \cref{AC with respect ot general hausdorff} is part of the following classical result (see \cite{MAT}*{Theorem 8.8}). 
\begin{lemma}
[Frostman]\label{l:frostman}
Let $A\subset \R^d$ be a Borel set and $s>0$. Then, the following facts are equivalent:
\begin{itemize}
    \item[(a)] $\mathcal{H}^s(A)>0$;
    \item[(b)] there exists a positive Borel measure $\mu$ such that $\mu(A)>0$ and 
    $$
    \mu(B_\delta(x))\lesssim \delta^s, \quad \forall x\in \R^d, \delta>0.
    $$
\end{itemize}
\end{lemma}

The full power of Frostman's result is not needed for our purposes. We remark that the non trivial part of \cref{l:frostman} is the implication $(a) \Rightarrow (b)$, while $(a) \Leftarrow (b)$ follows by the argument that we described in our more general anisotropic setting.

We mention also \cite{S05}*{Theorem 3.2}, which is closely related to our approach.

\begin{theorem}[Measure divergence vector fields]
\label{t:div_measure}
Let $\frac{d}{d-1}\leq r\leq \infty$ and $V\in L^r_{{\rm loc}}(\R^d)$ be a vector field such that $\div V$ is locally a signed bounded measure on $\R^d$. If $r <\infty$, it holds 
$$
\mathcal{H}^s(B)<\infty \Rightarrow |\div V|(B)=0, \quad \forall B \text{ bounded Borel set},
$$ 
where $s:= d-\frac{r}{r-1}$. If $r= \infty$ then $|\div V|\ll \mathcal{H}^{d-1}$, with $|\div V|\in \mathcal{M}_{{\rm loc}}(\R^d)$ the variation of the signed measure $\div V$. 
\end{theorem}

\begin{remark} A quick derivation of a slightly weaker result goes as follows. Assume that $\mu = \div{V}$ is a locally positive measure. Fix $x \in \R^d$ and let $\chi_\delta$ be a cut off function which localise $B_\delta (x)$, i.e. $\chi_\delta\in C^\infty_c(B_{2\delta}(x); [0,1])$ with $\chi_\delta\equiv 1$ on $B_\delta(x)$ and $|\nabla \chi_\delta|\lesssim \delta^{-1}$. Then, we estimate
\begin{equation*}
    \mu(B_\delta(x))\leq \int_{\R^d} \chi_\delta \,\rmd\mu = -\int_{\R^d} V\cdot \nabla \chi_\delta \, \rmd y\leq \|V\|_{L^r(B_{2\delta}(x))} \|\nabla \chi_\delta\|_{L^{r'}}\lesssim \omega_V(2\delta)\delta^{d\frac{r-1}{r}-1},
\end{equation*}
where we defined $\omega_V(\delta)=\sup_{x}\|V\|_{L^r(B_{\delta}(x))}$.
Then by \cref{l:frostman}, we infer that $\div V \ll \mathcal{H}^{\tilde s}$ with $\tilde s=d\frac{r-1}{r}-1=\frac{r-1}{r}s$, where $s$ is the exponent in \cref{t:div_measure}. In the limiting cases $r=\frac{d}{d-1}$ and $r=\infty$, this argument yields the same result as the \cref{t:div_measure}.
\end{remark}

One could directly apply the \cref{t:div_measure} to deduce results on general conservation laws of the form \eqref{conslaw} and \eqref{entropy_equality}, thereby proving \cref{t:main_general_laws}. However, since for general fluid equations the scaling is not isotropic in space and time, in the case of incompressible Euler (\cref{t:main_euler}) and incompressible Navier--Stokes (\cref{t:leray_hopf_intro}) we specialize the proof by considering the anisotropic Hausdorff measure as defined in \cref{d:hausdorff_meas}, yielding to sharper dimension bounds.  

Let us emphasize that cylinders asymptotic is a strictly stronger property with respect to the measure being absolutely continuous with respect to a reference Hausdorff. Indeed we have the following
\begin{proposition}
\label{p:counter_example} Let $B_1(0)\subset \R^d$ be the $d-$dimensional unit ball, $d\geq 2$. For all $\eps>0$ there exists a vector field $V_\eps\in L^{\frac{d}{d-1}}(B_1(0))$ such that $\div V\geq 0$ and $\div V_\eps\in L^1(B_1(0))$,  thus in particular $\div V_\eps\ll \mathcal{L}^d$,  but there exists a dimensional constant $c_d$ such that
\begin{equation}
\label{not_asymptotic}
\div V_\eps (B_\delta (0))= c_d \delta^\eps, \qquad \forall \delta<1.
\end{equation}
\end{proposition}
\begin{proof}
For $\eps>0$ set $V_\eps(x)=x |x|^{\eps-d}$. Clearly $V_\eps\in L^{\frac{d}{d-1}}(B_1(0))$ for all $\eps>0$.  To compute its distributional divergence we start by noticing that (outside the origin) we can perform the pointwise computation
\begin{equation}
\label{div_poitwise}
\div V_\eps(x)=\frac{\div x}{|x|^{d-\eps}}+x\cdot \nabla \left( \frac{1}{|x|^{d-\eps}}\right)=\frac{\eps}{|x|^{d-\eps}}, \qquad \forall x\neq 0.
\end{equation}
In fact, one can show that the previous formula gives the full distributional divergence in $B_1(0)$. Indeed, for any $\varphi \in C^\infty_c(B_1(0))$, we have
\begin{align}
\label{distr_div}
\left\langle \div V_\eps, \varphi \right\rangle &=-\int_{B_1(0)} V_\eps \cdot \nabla \varphi \, \rmd x = -\lim_{\kappa \rightarrow 0} \int_{B_1(0)\setminus B_\kappa (0)} V_\eps \cdot \nabla \varphi \, \rmd x \nonumber\\
&=-\lim_{\kappa \rightarrow 0} \int_{\partial B_\kappa (0)} \varphi V_\eps \cdot n  \, \rmd \mathcal{H}^{d-1}+\lim_{\kappa \rightarrow 0} \int_{B_1(0)\setminus B_\kappa (0)} \varphi  \div V_\eps  \, \rmd x,
\end{align}
where we denoted by $n:\partial B_\kappa (0)\rightarrow \mathbb{S}^{d-1}$ the inward unit normal. By \eqref{div_poitwise} we have 
$$
\lim_{\kappa \rightarrow 0} \int_{B_1(0)\setminus B_\kappa (0)}   \varphi  \div V_\eps \, \rmd x=\int_{B_1(0)} \frac{\eps}{|x|^{d-\eps}}  \varphi \, \rmd x,
$$
since $|x|^{\eps-d}\in L^1(B_1(0))$, and moreover
$$
\left| \lim_{\kappa \rightarrow 0} \int_{\partial B_\kappa (0)} \varphi V_\eps \cdot n  \, \rmd \mathcal{H}^{d-1}\right| \lesssim \lim_{\kappa \rightarrow 0} \kappa^\eps =0.
$$
Thus \eqref{distr_div} shows that the formula \eqref{div_poitwise} actually holds in $\mathcal{D}'(B_1(0))$. Finally,  denoting by $c_d=\mathcal{H}^{d-1}(B_1(0))$  the surface measure of the unit ball, we have
$$
\div V_\eps (B_\delta (0))=\int_{B_\delta(0))} \frac{\eps}{|x|^{d-\eps}}\, \rmd x=c_d \eps \int_0^\delta \frac{1}{r^{1-\eps}}\, \rmd r=c_d \delta^\eps.
$$
\end{proof}
The bound \eqref{not_asymptotic} shows that in general there are very regular measures (in terms of absolute continuity with respect to a reference measure) for which no local estimate on balls can hold. The (positive) measure constructed in the previous proposition is the divergence of a given vector field, which is indeed the relevant case in our analysis.

\subsection{Lower bounds for signed distributions}\label{s:signed_case} Here we give the details on how to prove dimension lower bounds on the support of nontrivial distributions $T$, thus not necessarily measures,  given by the divergence of a given vector field.  Of course in this case the lower bound does not follow from the local asymptotic on balls, and the corresponding absolute continuity with respect to $\mathcal H^s$.  The following proposition (and its proof) should be seen as an adaptation of \cref{t:div_measure}.
\begin{proposition}
Let $d\geq 2$,  $\frac{d}{d-1}\leq r\leq \infty$ and $V\in L^r_{\rm loc}(\R^d)$ be a $d-$dimensional vector field.  Denoting by $T:=\div V\in \mathcal{D}'(\R^d)$ and $s=d-\frac{r}{r-1}$, the following holds true.
\begin{itemize}
\item[(i)] If $r<\infty$ then $\mathcal{H}^s(\spt T)<\infty$ implies that $T\equiv 0$;
\item[(ii)] If $r=\infty$ then $\mathcal{H}^{d-1}(\spt T)=0$ implies that $T\equiv 0$.
\end{itemize}
In particular, if $T\not \equiv 0$ then necessarily $\dim_{\mathcal H} (\spt T)\geq s$.
\end{proposition}
\begin{proof}
We start by proving $(i)$. Fix $\eps>0$, $\varphi \in C^\infty_c(\R^d)$ and assume that $\mathcal H^s(\spt T)<\infty$. We will prove that 
\begin{equation}
\label{T_vanish}
| \langle T,\varphi \rangle |\leq\eps.
\end{equation}
For $\delta >0$ we can find a countable covering $\{ B_{r_i}\}_{i}$, $r_i<\delta$, of $\spt T$ such that 
\begin{equation*}
\sum_i r_i^s\leq 1+\mathcal H^s(\spt T).
\end{equation*}
In particular
\begin{equation}\label{Hd_small}
\mathcal{H}^d\left(\bigcup_i B_{r_i} \right)\leq \sum_i r_i^d\leq \delta^{d-s} \left(1+\mathcal H^s(\spt T) \right)
\end{equation}
can be made arbitrarily small if $\delta$ is chosen appropriately. Let $\tilde \eps>0$ be a positive constant which will be fixed at the end of the proof.  By \eqref{Hd_small} and the absolute continuity of the Lebesgue integral, if $\delta$ is small enough, we can guarantee
\begin{equation}
\label{norm_V_small}
\|V\|_{L^r\left(\bigcup_i B_{r_i} \cap \spt \varphi\right)}\leq \tilde \eps.
\end{equation}
In each ball consider $\chi_i\in C^\infty_c(B_{2r_i})$ such that $|\chi_i|\leq 1$,
\begin{equation}
\label{cutoff_bounds}
\chi_i\big|_{B_{r_i}}\equiv 1\quad \text{ and } \quad |\nabla \chi_i|\lesssim \frac{1}{r_i},
\end{equation}
and define $\chi:=\max_i \chi_i$. Clearly $\chi\in \Lip(\R^d)$,  and $\chi\equiv 1$ on $\spt T$, thus we can compute
$$
| \langle T,\varphi \rangle |=| \langle T,\chi\varphi \rangle |\leq\left| \int_{\R^d} V\cdot \nabla \varphi \chi\, \rm d x \right| + \left|\int_{\R^d}  V\cdot \nabla \chi \varphi\, \rm d x\right|=: I+II.
$$
By \eqref{Hd_small}, and since $|\chi|\leq 1$, we get  
$$
I\leq \|\nabla \varphi\|_{L^\infty} \|V\|_{L^r(\spt \varphi)} \|\chi\|_{L^\frac{r}{r-1}}\lesssim \left( \mathcal H^d\left(\bigcup_i B_{r_i}\right)\right)^\frac{r-1}{r}\lesssim \delta^{(d-s)\frac{r-1}{r}}.
$$
Moreover by   \eqref{norm_V_small} and \eqref{cutoff_bounds} we can estimate
\begin{align*}
II&\leq \|\varphi \|_{L^\infty} \|V\|_{L^r\left(\bigcup_i B_{r_i} \cap \spt \varphi\right)}\|\nabla \chi\|_{L^\frac{r}{r-1}}\lesssim \tilde \eps \left(\sum_i \int_{\R^d} |\nabla \chi_i|^\frac{r}{r-1}\,\rm d x \right)^\frac{r-1}{r}\\
&\lesssim \tilde \eps \left( \sum_i r_i^{d-\frac{r}{r-1}}\right)^\frac{r-1}{r}\lesssim \tilde \eps.
\end{align*}
Thus one can choose $\tilde \eps$, and then $\delta$, sufficiently small (depending on $\eps$) and conclude the validity of \eqref{T_vanish}.\\
\\
In case $(ii)$, i.e. when $r=\infty$, if $\mathcal H^{d-1}(\spt T)=0$, for $\delta>0$, we can find a countable covering of $\spt T$ such that $r_i<\delta$ and 
$$
\sum_i r_i^{d-1}\leq \tilde \eps.
$$
In this case we estimate
$$
I\leq \|\nabla \varphi \|_{L^\infty}\|V\|_{L^\infty(\spt \varphi)} \|  \chi\|_{L^1} \lesssim \sum_i r_i^d\lesssim \delta\tilde \eps
$$
and 
$$
II\leq \|\varphi\|_{L^\infty}\|V\|_{L^\infty(\spt \varphi)} \|  \nabla \chi\|_{L^1} \lesssim \sum_i r_i^{d-1}\lesssim \tilde \eps.
$$
As before, by choosing $\tilde \eps$ sufficiently small (here any $\delta<1$ is enough) we achieve \eqref{T_vanish}, which concludes the proof. 
\end{proof}

\section{Dissipation regularity for Euler}\label{s:e_proofs}

Here we give the more general version of \cref{t:main_euler}.

\begin{theorem}\label{t:euler_general}
Let $\Omega\subset \R^d$ be a domain. For  $\alpha>0$ and $3\leq r,q\leq \infty$ set
\begin{equation}\label{s:ugly_euler}
s = \min\left(d\frac{r-2}{r}-\alpha\frac{2}{q},d\frac{r-3}{r}-1+\alpha\frac{q-3}{q} \right).
\end{equation}
Assume $s\geq 0$ and that $u\in L_{{\rm loc}}^q(0,T;L^r_{{\rm loc}}(\Omega))$ and $p\in L^\frac{q}{2}_{{\rm loc}}(0,T;L^\frac{r}{2}_{{\rm loc}}(\Omega))$ is a weak solution of the incompressible Euler equations \eqref{E} having non-negative local anomalous dissipation measure $\varepsilon[u]\in \mathcal{M}_{{\rm loc}}(\Omega\times (0,T)) $. The following results hold true.
\begin{itemize}
    \item[(i)] If $q<\infty$, then for every Borel set $B\subset\joinrel\subset \Omega\times (0,T)$ we have that
    $$
    \mathcal{H}^s_\alpha(B)<\infty \, \Rightarrow \, \eps[u](B)=0,
    $$
    with the convention that $s=\min\left(d-\alpha\frac{2}{q},d-1+\alpha\frac{q-3}{q} \right)$ whenever $r=\infty$.
    \item[(ii)] If $q=\infty$, then
    $$
    \eps[u]\ll \mathcal{H}^s_\alpha,\quad \text{for} \quad s=\min\left(d\frac{r-2}{r},d\frac{r-3}{r}-1+\alpha \right),
    $$
    with the convention that $s=\min\left(d,d-1+\alpha \right)$ whenever $r=\infty$.
\end{itemize}
In particular, if $\eps[u]\not\equiv 0$ and concentrated on a space-time set $S$, it must hold
$$
\dim_{\mathcal{H}_\alpha}S \geq s.
$$
Further, if $\Omega$ is bounded and $(u,p)\in L^q(0,T; L^{r}(\Omega))\times L^{\frac{q}{2}}(0,T; L^{\frac{r}{2}}(\Omega))$ has a boundary extendable dissipation $\overline \eps[u]\in \mathcal{M}_{{\rm loc}}(\overline \Omega\times (0,T])$ in the sense of \cref{d:extendable_diss}, then the same conclusions $(i)$ and $(ii)$ above hold for $\overline \eps[u]$, where in $(i)$ every Borel set $B\subset\joinrel\subset \overline \Omega \times (0,T]$ is allowed.
\end{theorem}
It is clear that in order to optimize the dimension bound $s$ in \cref{t:euler_general}, one has to choose $\alpha$ such that the two terms in the minimum in \eqref{s:ugly_euler} balance. This happens if and only if 
\begin{equation}
\alpha_{opt}=\frac{q}{q-1}\frac{r+d}{r},\label{alpha_opt}
\end{equation}
which in particular proves \cref{t:main_euler}.

Thanks to \cref{AC with respect ot general hausdorff} and \cref{c:lowerbound_support}, \cref{t:euler_general} is a direct consequence of the following proposition. We adopt the cylinders notation $\mathcal{C}^\alpha_{\delta}(x,t)$ from \cref{s:tools}. If $x\in \partial \Omega$ and/or $t=T$,  i.e. $(x,t)\in \overline \Omega\times (0,T]$, we assume that the cylinders are implicitly restricted to the space-time domain.

\begin{proposition}\label{p_euler_balls}
Let $\Omega\subset \R^d$ be a domain. For  $\alpha>0$ and $3\leq r,q\leq \infty$ define $s$ as in \eqref{s:ugly_euler}.
Assume $s\geq 0$ and that $u\in L_{{\rm loc}}^q(0,T;L^r_{{\rm loc}}(\Omega))$ and $p\in L^\frac{q}{2}_{{\rm loc}}(0,T;L^\frac{r}{2}_{{\rm loc}}(\Omega))$ is a weak solution of the incompressible Euler equations \eqref{E} having non-negative local anomalous dissipation measure $\varepsilon[u]\in \mathcal{M}_{{\rm loc}}(\Omega\times (0,T)) $. Let $K\subset\joinrel\subset \Omega\times (0,T)$ be a compact set. Fix $\delta_0>0$ such that
$$\mathcal{C}^\alpha_{\delta} (x,t) \subset\joinrel\subset  \Omega\times (0,T) \ \ \ \forall (x,t) \in K, \ \forall \delta \leq \delta_0. $$
Then, for any $(x,t)\in K$ the following results hold true.
\begin{itemize}
    \item[(i)] If $q<\infty$, then there exists a modulus of continuity $\omega:[0,\infty)\rightarrow [0,\infty)$ such that we have
    $$
    \eps [u](\mathcal{C}^\alpha_{\delta}(x,t))\leq \omega(\delta) \delta^s \quad \forall \delta \leq \frac{\delta_0}{4},
    $$
    with the convention that $s = \min\left(d-\alpha\frac{2}{q},d-1+\alpha\frac{q-3}{q} \right)$ whenever $r=\infty$.
    \item[(ii)] If $q=\infty$, then
    $$
    \eps [u](\mathcal{C}^\alpha_{\delta}(x,t))\lesssim \delta^s \quad   \forall \delta \leq \frac{\delta_0}{4}, \quad \text{where} \quad s=\min\left(d\frac{r-2}{r},d\frac{r-3}{r}-1+\alpha \right),
    $$
   with the convention that $s=\min\left(d,d-1+\alpha \right)$ whenever $r=\infty$.
\end{itemize}
Further, if $\Omega$ is bounded and $(u,p)\in L^q(0,T; L^{r}(\Omega))\times L^{\frac{q}{2}}(0,T; L^{\frac{r}{2}}(\Omega))$ has a boundary extendable dissipation $\overline \eps[u]\in \mathcal{M}_{{\rm loc}}(\overline \Omega\times (0,T])$ in the sense of \cref{d:extendable_diss}, then the same conclusions in $(i)$ and $(ii)$ above hold for $\overline \eps[u]$, for every $(x,t)\in \overline \Omega \times (0,T]$.
\end{proposition}
\begin{proof}
We only give the full details in the case $\eps[u]\in \mathcal{M}_{{\rm loc}}(\Omega\times (0,T))$. Then, the reader can reconstruct the proof when $u$ has a boundary extendable dissipation $\overline \eps[u]\in \mathcal{M}_{{\rm loc}}(\overline \Omega\times (0,T])$ by simply allowing boundary points $(x,t)\in \overline \Omega \times (0,T]$. Indeed, the left hand side in the boundary extended energy balance \eqref{DR_measure_boundary} stays the same.

Fix $(x,t)\in K $ and let $\chi_\delta$ and $\eta_\delta$ be two cut off functions which localise $B_\delta (x)$ and $(t-\delta^\alpha,t+\delta^\alpha)$ respectively, i.e. $\chi_\delta\in C^\infty_c(B_{2\delta}(x); [0,1])$ with $\chi_\delta\equiv 1$ on $B_\delta(x)$ and $\eta_\delta\in C^\infty_c((t-(2\delta)^\alpha,t+(2\delta)^\alpha);[0,1])$ such that $\eta_\delta\equiv 1$ on $(t-\delta^\alpha,t+\delta^\alpha)$. Moreover, we choose the cut off functions such that $|\nabla \chi_\delta|\lesssim \delta^{-1}$ and $| \eta'_\delta|\lesssim \delta^{-\alpha}$. Here, $\delta$ is assumed to be small enough such that $\mathcal{C}^\alpha_{\delta}(x,t) \subset\joinrel\subset \Omega \times (0,T)$. Then, testing \eqref{DR_measure_inter} with $\chi_\delta\eta_\delta$, and since $\eps[u]\geq 0$, yields
$$
\eps[u](\mathcal{C}^\alpha_{\delta}(x,t)) \leq \int \frac{|u|^2}{2} \chi_\delta \eta'_\delta + \int \frac{|u|^2}{2} u\cdot \nabla \chi_\delta \eta_\delta + \int p u \cdot \nabla \chi_\delta \eta_\delta = I_{\delta}+II_{\delta} + III_{\delta}.
$$
We estimate separately the three terms. By using H\"older's inequality in space with exponents $\frac{r}{2}, \frac{r}{r-2}$ and in time with exponents $\frac{q}{2}, \frac{q}{q-2}$, we have that 
\begin{align}\label{est_I}
    |I_{\delta}| & \leq\int_0^T \int_{\Omega}|\eta_\delta'| |u|^2 |\chi_\delta |\, \rmd y\, \rmd \tau\leq \int_0^T |\eta_\delta'(\tau)| \|u(\tau)\|_{L^{r}(B_{2\delta}(x))}^2 \|\chi_\delta\|_{L^{\frac{r}{r-2}}}\, \rmd \tau\nonumber \\
    &\lesssim \delta^{d\frac{r-2}{r}}\int_0^T |\eta_\delta' (\tau)| \|u(\tau)\|^2_{L^{r}(B_{2\delta}(x))} \, \rmd \tau \leq \delta^{d\frac{r-2}{r}}\|u\|^2_{L^qL^r(\mathcal{C}^\alpha_{2\delta}(x,t))} \|\eta_\delta'\|_{L^{\frac{q}{q-2}}}\nonumber \\
    &\lesssim \|u\|^2_{L^qL^r(\mathcal{C}^\alpha_{2\delta}(x,t))}\delta^{d\frac{r-2}{r}-\alpha\frac{2}{q}}.
\end{align}
To estimate the second term, we use H\"older's inequality in space with exponents $\frac{r}{3}, \frac{r}{r-3}$ and in time with exponents $\frac{q}{3}, \frac{q}{q-3}$. Thus, we have
\begin{align}\label{est_II}
    |II_{\delta}| & \leq \int_0^T \int_{\Omega}|\eta_\delta| |u|^3 | \nabla \chi_\delta |\, \rmd y\, \rmd \tau\leq \int_0^T |\eta_\delta(\tau)| \|u(\tau)\|_{L^{r}(B_{2\delta}(x))}^3 \| \nabla \chi_\delta\|_{L^{\frac{r}{r-3}}}\, \rmd \tau\nonumber \\
    &\lesssim \delta^{d\frac{r-3}{r} - 1}\int_0^T |\eta_\delta (\tau)| \|u(\tau)\|^3_{L^{r}(B_{2\delta}(x))} \,\rmd \tau\leq \delta^{d\frac{r-3}{r} -1}\|u\|^3_{L^qL^r(\mathcal{C}^\alpha_{2\delta}(x,t))} \|\eta_\delta \|_{L^{\frac{q}{q-3}}}\nonumber \\
    &\lesssim \|u\|^3_{L^qL^r(\mathcal{C}^\alpha_{2\delta}(x,t))}\delta^{d\frac{r-3}{r} -1+\alpha\frac{q-3}{q}}.
\end{align}





To estimate the third term we use H\"older's inequality in space with exponents $\frac{r}{2}, r , \frac{r}{r-3}$ and in time with exponents $ \frac{q}{q-3}, \frac{q}{2}, q$. Thus, we obtain 
\begin{align}\label{est_III}
    |III_{\delta}| & \leq \int_0^T \int_{\Omega}|\eta_\delta| |p | |u| | \nabla \chi_\delta | \,\rmd y\,\rmd \tau\nonumber\\
    &\leq \int_0^T |\eta_\delta(\tau)| \| p(\tau)\|_{L^{\frac{r}{2}}(B_{2\delta}(x))} \|u(\tau)\|_{L^{r}(B_{2\delta}(x))}  \| \nabla \chi_\delta\|_{L^{\frac{r}{r-3}}}\, \rmd \tau\nonumber \\
    & \lesssim \delta^{d\frac{r-3}{r} - 1}\int_0^T |\eta_\delta (\tau)| \| p(\tau)\|_{L^{\frac{r}{2}}(B_{2\delta}(x))} \|u(\tau)\|_{L^{r}(B_{2\delta}(x))} \,\rmd \tau \nonumber
    \\ & \lesssim \delta^{d\frac{r-3}{r} -1}\|p\|_{L^\frac{q}{2} L^\frac{r}{2}(\mathcal{C}^\alpha_{2\delta}(x,t))} \|u\|_{L^q L^r(\mathcal{C}^\alpha_{2\delta}(x,t))} \|\eta_\delta \|_{L^{\frac{q}{q-3}}}\nonumber \\
    &\lesssim \|p\|_{L^\frac{q}{2} L^\frac{r}{2}(\mathcal{C}^\alpha_{2\delta}(x,t))}\|u\|_{L^q L^r(\mathcal{C}^\alpha_{2\delta}(x,t))} \delta^{d\frac{r-3}{r} -1+\alpha\frac{q-3}{q}}.
\end{align}
To resume, by \eqref{est_I}, \eqref{est_II} and \eqref{est_III}, we achieve
\begin{equation}\label{almost_final_est}
\eps[u](\mathcal{C}^\alpha_{\delta}(x,t))\lesssim \omega(\delta)\left( \delta^{d\frac{r-2}{r}-\alpha\frac{2}{q} } + \delta^{d\frac{r-3}{r}-1+\alpha\frac{q-3}{q}}\right),
\end{equation}
where we set
$$
\omega(\delta)=\sup_{(x,t)\in K}\left(\|u\|^2_{L^qL^r(\mathcal{C}^\alpha_{2\delta}(x,t))}+\|u\|^3_{L^qL^r(\mathcal{C}^\alpha_{2\delta}(x,t))}+\|p\|_{L^\frac{q}{2} L^\frac{r}{2}(\mathcal{C}^\alpha_{2\delta}(x,t))}\|u\|_{L^q L^r(\mathcal{C}^\alpha_{2\delta}(x,t))}\right).
$$
Hence, by choosing $s$ as in \eqref{s:ugly_euler} we conclude
\begin{equation*}
    \eps[u](\mathcal{C}^\alpha_{\delta}(x,t))\lesssim \omega(\delta)\delta^{s} \ \ \ (x,t)\in K, \ \delta \leq \frac{\delta_0}{4}.   
\end{equation*}
By absolute continuity of the Lebesgue integral it is clear that $\omega(\delta)$ gives a modulus of continuity as soon as $q<\infty$. Thus, the proof is concluded.
\end{proof}

\begin{remark}[Euler scaling invariant norms]
    \label{r:scaling_euler}
    The Euler equations are invariant under the scaling
    \begin{equation*}
        u_\lambda (x,t)=\lambda^{\alpha-1}u(\lambda x, \lambda^\alpha t) \quad \text{and} \quad  p_\lambda (x,t)=\lambda^{2(\alpha-1)}p(\lambda x, \lambda^\alpha t),
    \end{equation*}
    for all values of $\alpha$. By direct computation, we have 
    $$
    \| u_\lambda\|_{L^q_tL^r_x}=\lambda^{-\frac{d+r}{r}+\alpha\frac{q-1}{q}} \| u\|_{L^q_tL^r_x},
    $$
    giving a scaling invariant norm if and only if $\alpha$ is chosen as in \eqref{alpha_opt}. Thus, optimizing the dimension bound of the previous \cref{t:euler_general} in terms of $\alpha$ coincides with having the corresponding $L^q_tL^r_x$ norm scale invariant. 
\end{remark}

\subsection{Some numerology in relevant cases}\label{s:numerology_euler}
In this section, we report some explicit bounds given by \cref{t:main_euler} in cases that play a privileged role for turbulence. In what follows, we assume that $\Omega=\R^d$ or $\Omega=\T^d$ and $u$ is a weak solution to the Euler equations \eqref{E} such that $\e[u]$ is a non-negative measure. Hence, the hypothesis on the pressure follows by that on the velocity. In general bounded domains the same conclusions holds true, provided that we assume the right integrability on the pressure.

First, the case of $L^\infty(\Omega\times (0,T))$ solutions, giving $s=d$ as space-time Hausdorff dimension of the (non-trivial) dissipation, has been already discussed throughout the introduction. In addition, we remark that this result loosely complements that of Dubrulle  \& Gibbon \cites{DG22,Gibbon22}. The authors showed that $C(h)\geq1-3h$, where $C(h)$ is the codimension of the set having H\"{o}lder exponent $h$. In view of the $4/5$ths law, our result indicates that $C(h_*)\leq 1$, where $h_*$ is the H\"{o}lder exponent corresponding to $\zeta_3=1$ within the multifractal formalism. However, since we are directly dealing with the viscous dissipation measure, the relevant multifractal spectrum is instead the $f(\alpha)$ spectrum discussed by Meneveau and Sreenivasan \cites{Meneveau87,Meneveau88,Meneveau91}. We now discuss other cases.

\begin{itemize}
    
    \item[Case 1:] Assume that $u \in L^\infty(0,T;L^{r}(\Omega))$ for every $r \in [3, 3d/(d-1)]$.  
    Thus, in the conclusion of \cref{t:main_euler} we get 
    $$\alpha(d,r) = 1+\frac{d}{r}, \ \ \ s(d,r)=d-\frac{2d}{r}.$$
    In terms of dimension lower bounds the strongest estimate would be for $\alpha$ smaller and $s$ bigger: $\alpha$ close to $1$, i.e. $\alpha$ smallest, means that we look at the most isotropic measure, while $s$ largest means that we get a larger dimension bound. Hence, we wish to maximise $s$ and minimize $\alpha$, which is equivalent to choose $r_{opt} = \frac{3d}{d-1}$. 
    Then, we obtain 
    \begin{equation}
        \overline{\alpha}=\alpha(d,r_{opt}) = \frac{2+d}{3}, \ \ \ \overline{s} = s(d,r_{opt}) = \frac{2+d}{3}. \label{optimal_choice 2}
    \end{equation}
    In particular, if $\e[u]$ is non-trivial, we have the dimension lower bound 
    \begin{equation}
        \dim_{\mathcal{H}^{\overline{\alpha}}}\big(\spt{\e[u]}\big) \geq \overline{s}. \label{optimal_lower_bound 2}
    \end{equation}

    \item[Case 1bis: ] Assume that $u \in L^\infty(0,T; B^\frac13_{3,\infty}(\Omega))$. By Sobolev embedding, we only have that $u\in L^\infty(0,T;L^r(\Omega))$ for all $r\in [3,3d/(d-1))$, thus the limiting case $r=3d/(d-1)$ is not included. However, by a minor modification of the proof of \cref{p_euler_balls}, we obtain that the dimensional lower bound \eqref{optimal_lower_bound 2} is still valid, with $\overline{\alpha}, \overline{s}$ defined as in \eqref{optimal_choice 2}. Indeed, for every $r \in [3, 3d/(d-1))$, we choose $\alpha= \frac{2+d}{3}$ in place of \eqref{alpha_opt} and we modify \eqref{almost_final_est} accordingly. Thus, letting $\overline{\alpha}, \overline{s}$ be as in \eqref{optimal_choice 2}, we obtain that $\e[u] \ll \mathcal{H}_{\overline{\alpha}}^{\overline{s}-\gamma_r}$, for some $\gamma_r>0$ such that $\gamma_r \to 0$ as $r \to 3d/(d-1)$. If $\e[u]$ is non trivial, letting $r \to 3d/(d-1)$, we obtain the dimensional lower bound \eqref{optimal_lower_bound 2} (see \cref{general hausdorff dimension}).

    \item[Case 2:] Assume $u \in L^\infty(0,T;H^\beta(\Omega))$, with $H^\beta$ the usual $L^2-$based Sobolev space, $\beta>0$. Let us restrict to the case $\beta\in (0,5/6)$, since for $\beta\geq 5/6$ it is known that $\eps[u]\equiv 0$, see \cite{CCFS08}. We have the Sobolev embedding $H^\beta(\Omega)\subset L^{\frac{2d}{d-2\beta}}(\Omega)$ for all $\beta<d/2$ (and, in particular, for all $\beta<5/6$). To apply \cref{t:main_euler} we need $r\geq 3$, which in this case reads as $\beta\geq d/6$. Thus, in \cref{t:main_euler} we obtain 
    $$\alpha(d,\beta) = 1+\frac{d-2\beta}{2}, \ \ \ s(d,\beta)=s(\beta)=2\beta,$$
    yielding a non-trivial conclusion in the  range $\beta\in(d/6,5/6)$, if $d<5$, since $s$ was required to be non-negative. Curiously, the dimension bound in this case does not depend on the space dimension $d$, but instead only on the second--order structure exponent $\beta$.   
\end{itemize}

\section{Dissipation regularity for incompressible Navier--Stokes}\label{s:ns_proofs}

Here we give the more general version of \cref{t:leray_hopf_intro} by using the measure $\mathcal{H}^s_\alpha$, keeping $\alpha$ free.  To lighten the notation we simply write $u$ instead of $u^\nu$, since the viscosity parameter is fixed. 

\begin{theorem}\label{t:leray-hopf_general}
Let $\Omega\subset \R^d$ be a domain and let $u\in L^\infty(0,T;L^2(\Omega))\cap L^2(0,T;H^1(\Omega))$ be a Leray--Hopf weak solution to the incompressible Navier--Stokes equations \eqref{NS} having non-negative local dissipation measure $D[u]\in \mathcal{M}_{{\rm loc}}(\Omega\times (0,T)) $. For $\alpha>0$ and $3\leq r,q\leq \infty$ define
\begin{equation}\label{ugly_s}
s=\min\left(  d\frac{r-2}{r}-\alpha\frac{2}{q},-1+d\frac{r-3}{r}+\alpha\frac{q-3}{q},-2+d\frac{r-2}{r}+\alpha\frac{q-2}{q}\right).
\end{equation}
Assume that $s\geq0$, $u\in L_{{\rm loc}}^q(0,T;L_{{\rm loc}}^r(\Omega))$ and $p\in L_{{\rm loc}}^\frac{q}{2}(0,T;L_{{\rm loc}}^\frac{r}{2}(\Omega))$.
\begin{itemize}
    \item[(i)] If $q<\infty$, then for every Borel set $B\subset\joinrel\subset \Omega\times (0,T)$ we have
    $$
    \mathcal{H}^s_\alpha(B)<\infty \, \Rightarrow \, D[u](B)=0,
    $$
    with the convention that $s=\min\left( d-\alpha\frac{2}{q},d-1+\alpha\frac{q-3}{q},d-2+\alpha\frac{q-2}{q}\right)$ whenever $r=\infty$.
    \item[(ii)] If $q=\infty$, then
    \begin{equation}\label{ugly_s_2}
    D[u]\ll \mathcal{H}^s_\alpha\qquad \text{for} \quad s=\min\left(  d\frac{r-2}{r},-1+d\frac{r-3}{r}+\alpha,-2+d\frac{r-2}{r}+\alpha\right),
     \end{equation}
     with the convention that $s=\min\left( d,d-2+\alpha \right)$ whenever $r=\infty$.
\end{itemize}
In particular, if $D[u]\not\equiv 0$ and concentrated on a space-time set $S$, it must hold
$$
\dim_{\mathcal{H}_\alpha}S\geq s.
$$
\end{theorem}
Thus, \cref{t:leray_hopf_intro} directly follows by choosing $\alpha=2$, since $\mathcal{H}^s_2$ is the parabolic Hausdorff measure $\mathcal{P}^s$.  
Thanks to \cref{AC with respect ot general hausdorff} and \cref{c:lowerbound_support}, \cref{t:leray-hopf_general} is a direct consequence of the following proposition, where we adopt the cylinder notation $\mathcal{C}^\alpha_{\delta}(x,t)=B_\delta(x)\times (t-\delta^\alpha,t+\delta^\alpha)$ introduced in \cref{s:tools}. Note that for Leray--Hopf weak solutions $|\nabla u|^2\in L^1(\Omega\times (0,T))$. Thus, $|\nabla u|^2$ defines a positive measure absolutely continuous with respect to $\mathcal{L}^{d+1}$, with $\mathcal{L}^{d+1}$ the $d+1$ space-time Lebesgue measure.

\begin{proposition}\label{p:leray-hopf_balls}
Let $\Omega\subset \R^d$ be a domain and let $u\in L^\infty(0,T;L^2(\Omega))\cap L^2(0,T;H^1(\Omega))$ be a Leray--Hopf weak solution to the incompressible Navier--Stokes equations \eqref{NS} having non-negative local dissipation measure $D[u]\in \mathcal{M}_{{\rm loc}}(\Omega\times (0,T)) $. Let $K\subset\joinrel\subset \Omega\times (0,T)$ be a compact set.  Fix $\delta_0>0$ such that
$$\mathcal{C}^\alpha_{\delta}(x,t) \subset\joinrel\subset  \Omega\times (0,T) \ \ \ \forall (x,t) \in K, \ \forall \delta \leq \delta_0. $$
For all $\alpha>0$ and $3\leq r,q\leq \infty$ define $s$ as in \eqref{ugly_s}.
Assume that $s\geq0$, $u\in L_{{\rm loc}}^q(0,T;L_{{\rm loc}}^r(\Omega))$ and $p\in L_{{\rm loc}}^\frac{q}{2}(0,T;L_{{\rm loc}}^\frac{r}{2}(\Omega))$.  
Then, for every $(x,t)\in K$, and letting $\mu\in \mathcal{M}_{{\rm loc}}(\Omega\times (0,T))$ the positive measure $\mu=D[u]+\nu|\nabla u|^2$, the following results hold true.
\begin{itemize}
    \item[(i)] If $q<\infty$, then there exists a modulus of continuity $\omega:[0,\infty)\rightarrow [0,\infty)$ such that we have
    \begin{equation*}
  \mu\left(\mathcal{C}^\alpha_{\delta}(x,t)\right)\leq \omega(\delta) \delta^s \quad \forall \delta \leq \frac{\delta_0}{4},
    \end{equation*}
    with the convention that $s=\min\left( d-\alpha\frac{2}{q},d-1+\alpha\frac{q-3}{q},d-2+\alpha\frac{q-2}{q}\right)$ whenever $r=\infty$.
    \item[(ii)] If $q=\infty$, then for $s$ as in \eqref{ugly_s_2} it holds
    $$ \mu\left(\mathcal{C}^\alpha_{\delta}(x,t)\right)\lesssim \delta^s \quad   \forall \delta\leq\frac{\delta_0}{4},
    $$
    with the convention that $s=\min\left( d, d-2+\alpha\right)$ whenever $r=\infty$.
\end{itemize}
In particular, since $D[u]\geq 0$, we have
\begin{equation}\label{morrey-type_grad}
 \int_{\mathcal{C}^\alpha_{\delta}(x,t)}|\nabla u|^2\,dy\,d\tau\lesssim  \delta^s \quad \forall \delta \leq \frac{\delta_0}{4}.
    \end{equation}
\end{proposition}

\begin{proof}
Set $\mu= D^u + \nu\abs{\nabla u}^2$ and fix a point $(x,t) \in K$. For any $\delta>0$  chose two cutoff functions, one in space and one in time, $\chi_\delta$ and $\eta_\delta$ which localise  $B_\delta (x)$ and $(t-\delta^\alpha,t+\delta^\alpha)$ respectively, i.e. $\chi_\delta\in C^\infty_c(B_{2\delta}(x);[0,1])$ with $\chi_\delta\equiv 1$ on $B_\delta(x)$ and $\eta_\delta\in C^\infty_c((t-(2\delta)^\alpha,t+(2\delta)^\alpha);[0,1])$ such that $\eta_\delta\equiv 1$ on $(t-\delta^\alpha,t+\delta^\alpha)$. Moreover, choose the cut off functions such that $|\nabla \chi_\delta|\lesssim \delta^{-1}$ and $| \eta'_\delta|\lesssim \delta^{-\alpha}$. Since $\mu\geq 0$, by testing \eqref{nsebal} with $\chi_\delta\eta_\delta$ we get for the cylinder $\mathcal{C}^\alpha_{\delta}(x,t)$ the following estimate
\begin{align*}
\mu\left(\mathcal{C}^\alpha_{\delta}(x,t)\right) & \leq \int \frac{|u|^2}{2} \chi_\delta \eta'_\delta + \int \frac{|u|^2}{2}  u\cdot \nabla \chi_\delta \eta_\delta + \int p u \cdot \nabla \chi_{\delta} \eta_\delta  + \nu \int \frac{\abs{u}^2}{2} \eta_\delta \Delta\chi_\delta
 \\ & =I_{\delta}+II_{\delta} + III_{\delta} + IV_{\delta}.   
\end{align*}
We neglect multiplicative constants. Notice that in the proof of \cref{p_euler_balls} we have already estimated the terms $I_\delta, II_\delta,III_\delta$. Indeed, recalling \eqref{est_I}, \eqref{est_II} and \eqref{est_III}, we have
\begin{align*}
    |I_\delta|&\leq \int_{0}^T \int_{\Omega} \frac{\abs{u}^2}{2} \abs{\chi_\delta} \abs{\eta'_\delta}\, \rmd y \, \rmd \tau \lesssim \|u\|_{L^qL^r(\mathcal{C}^\alpha_{2\delta}(x,t))}^2 \delta^{d \frac{r-2}{r} -\alpha\frac{2}{q}},\\
   |II_\delta|&\leq \int_0^T \int_{\Omega} \frac{\abs{u}^3}{2} \abs{\nabla \chi_\delta} \abs{\eta_\delta}\, \rmd y \, \rmd \tau  \lesssim \| u \|_{L^qL^r(\mathcal{C}^\alpha_{2\delta}(x,t))}^3 \delta^{-1 + d \frac{r-3}{r}+\alpha\frac{q-3}{q}},\\
    |III_\delta|&\leq \int_0^T \int_{\Omega}|\eta_\delta| |p | |u| | \nabla \chi_\delta | \,\rmd y\,\rmd \tau  \lesssim \|p\|_{L^\frac{q}{2} L^\frac{r}{2}(\mathcal{C}^\alpha_{2\delta}(x,t))}\|u\|_{L^q L^r(\mathcal{C}^\alpha_{2\delta}(x,t))}\delta^{-1 + d \frac{r-3}{r}+\alpha\frac{q-3}{q}}.
\end{align*}
Hence, we are left with $IV_\delta$. Using H\"older's inequality in space with exponents $\frac{r}{2}, \frac{r}{r-2}$ and in time with exponents $\frac{q}{2}, \frac{q}{q-2}$, we infer that 
\begin{align*}
    \int_0^T \int_{\Omega} \frac{\abs{u}^2}{2} \abs{\Delta\chi_\delta} \abs{\eta_\delta}\, \rmd y \, \rmd \tau & \lesssim \int_0^T \abs{\eta_\delta(\tau)} \| u(\tau) \|_{L^r(B_{2\delta}(x))}^2 \| \Delta \chi_\delta \|_{L^\frac{r}{r-2}} \, \rmd \tau
    \\ & \lesssim \delta^{-2 + d\frac{r-2}{r}} \int_0^T\abs{\eta_\delta} \| u(\tau) \|_{L^r(B_{2\delta}(x))}^2 \, \rmd \tau
    \\ & \lesssim \delta^{-2 + d\frac{r-2}{r}} \| u\|_{L^qL^r(\mathcal{C}^\alpha_{2\delta}(x,t))}^2 \norm{\eta_\delta}_{L^\frac{q}{q-2}} 
    \\ & \lesssim \| u\|_{L^qL^r(\mathcal{C}^\alpha_{2\delta}(x,t))}^2 \delta^{-2 + d\frac{r-2}{r} +\alpha\frac{q-2}{q}}. 
\end{align*}
By putting all together we have shown that 
\begin{align*}
\mu\left(\mathcal{C}^\alpha_{\delta}(x,t)\right) \lesssim\omega(\delta)\delta^s,
\end{align*}
where we define 
$$
\omega(\delta)=\sup_{(x,t)\in K}\left(\| u\|_{L^qL^r(\mathcal{C}^\alpha_{2\delta}(x,t))}^2+\| u\|_{L^qL^r(\mathcal{C}^\alpha_{2\delta}(x,t))}^3+ \|p\|_{L^\frac{q}{2} L^\frac{r}{2}(\mathcal{C}^\alpha_{2\delta}(x,t))}\|u\|_{L^qL^r(\mathcal{C}^\alpha_{2\delta}(x,t))}\right),
$$
and $s$ is given by \eqref{ugly_s}.
By the absolute continuity of the Lebesgue integral it follows that $\omega(\delta)$ is a modulus of continuity as soon as $q<\infty$. This concludes the proof.
\end{proof}

\begin{remark} [Choices of $\alpha$ and scaling for Navier--Stokes]\label{r:scaling_NS}
    It is well known that the Navier--Stokes equations impose a diffusive scaling and Prodi--Serrin classes are the scaling invariant ones in this setting. However, \cref{t:leray-hopf_general} is empty in the Prodi--Serrin regime, since $D[u] \equiv 0$ by smoothness of solutions. 
 One can check that if $\alpha=2$ and $3 \leq q,r \leq +\infty$ satisfy the Prodi--Serrin condition $\frac{d}{r} + \frac{2}{q} = 1$, then the three terms in the minimum that defines $s$ in \eqref{ugly_s} equals $d-2$. In other words, in this setting, all the terms give exactly the same contribution. As already explained in the introduction, in the three-dimensional case this perfectly matches with the Caffarelli--Kohn--Nirenberg partial regularity \cite{CKN}, giving $D[u]\equiv 0$, consistent with the fact that such solutions are smooth.
Therefore, the interesting case is looking outside the Prodi--Serrin class.
Heuristically, looking at Leray--Hopf solutions outside the Prodi--Serrin regime we expect the contribution of the Laplacian to be negligible with respect to that of the Euler part. As already discussed in \cref{r:scaling_euler}, for any $1 \leq q,r \leq +\infty$ and for any $u \in L^q_t(L^r_x)$, letting $u_\lambda (x,t) = \lambda^{\alpha-1} u(\lambda x, \lambda^\alpha t)$, then $\|u_\lambda\|_{L^q_t L^r_x} = \| u \|_{L^q_t L^r_x}$ for any $\lambda$ if and only if $\alpha=\alpha_{opt}$ as defined in \eqref{alpha_opt}. One can also check that $\alpha_{opt} > 2$ if and only if $\frac{d}{r} + \frac{2}{q} > 1$ (i.e. outside Prodi--Serrin) and $\alpha_{opt} = 2$ if and only if $\frac{d}{r} + \frac{2}{q} = 1$. Thus, choosing $\alpha = \alpha_{opt}$ in \cref{t:leray-hopf_general} and assuming that $\frac{d}{r} + \frac{2}{q} \geq 1$, we discover that the first two terms in the minimum in \eqref{ugly_s} balance and the last term (the one coming from the Laplacian) is subordinate. This is consistent with the heuristic that the Laplacian term is lower order outside the regularity classes. More precisely, we have that 
    $$ d\frac{r-2}{r}-\alpha_{opt} \frac{2}{q} = -1+d\frac{r-3}{r}+\alpha_{opt}\frac{q-3}{q} = d \frac{r-2}{r} - \frac{2}{q-1} \frac{r+d}{r}, $$
    $$ -2 + d \frac{r-2}{r} + \alpha_{opt} \frac{q-2}{q} = -2 + d \frac{r-2}{r} + \frac{q-2}{q-1} \frac{r+d}{r} \geq d \frac{r-2}{r} - \frac{2}{q-1} \frac{r+d}{r}, $$
    being the latter inequality equivalent to $\frac{d}{r} + \frac{2}{q} \geq 1$. Thus, we find
    $$s = d \frac{r-2}{r} - \frac{2}{q-1} \frac{r+d}{r}. $$
    Hence, the choice of $\alpha = \alpha_{opt}$ appears natural in order to deal with Leray--Hopf weak solutions outside the Prodi--Serrin regime, while keeping the corresponding $L^q_tL^r_x$ norm scaling invariant. 
\end{remark}

\appendix
\section{Vanishing viscosity vs boundary dissipation}\label{app:vanish_visc}

We have the following result.

\begin{proposition} \label{p:boundary dissipation for vanishing visc}
Let $\Omega \subset \R^d$ be a bounded domain and let $\{(u^\nu, p^\nu)\}_{\nu>0}$ be a sequence of Leray--Hopf solutions to the incompressible Navier--Stokes system \eqref{NS} such that $u^\nu, p^\nu \in C^\infty(\overline \Omega \times (0,T])$ for every $\nu>0$. Assume that 
\begin{itemize}
    \item[(i)] $u^\nu \to u$ strongly in $L^3(\Omega\times (0,T))$ and $p^\nu \rightharpoonup p$ weakly in $L^\frac32 (\Omega\times (0,T))$;
    \item[(ii)] $\nu|\nabla u^\nu|^2\rightharpoonup \mu\in \mathcal{M}(\overline{\Omega}\times (0,T])$ and $|u^\nu(\cdot,T)|^2\rightharpoonup \lambda\in \mathcal{M}(\overline{\Omega})$ in the sense of measures.
\end{itemize}
Then, $(u,p)$ is a weak solution to the incompressible Euler system \eqref{E} with positive extendable boundary dissipation $\overline\e[u]$ according to \cref{d:extendable_diss}. Moreover, we have that
$$\overline\e[u] = \mu + \frac{\lambda}{2} \otimes \delta_T,$$
with $\delta_T\in \mathcal{M}([0,T])$ the Dirac delta in $t=T$. 
\end{proposition}
We remark that the assumption $(ii)$ always holds true, possibly up to subsequences, for Leray--Hopf solutions as soon as the sequence of initial data $\{ u_0^\nu\}_{\nu>0}$ is bounded in $L^2(\Omega)$.
\begin{proof} 
Since we have that $u^\nu \to u$ strongly in $L^2(\Omega\times (0,T))$, it is easy to prove that $(u,p)$ is a weak solution to the incompressible Euler system \eqref{E} in $\Omega \times (0,T)$ with impermeability boundary condition $u\cdot n=0$ on $\partial \Omega\times (0,T)$. Moreover, by $(i)$ we obviously have $u\in L^3(\Omega\times (0,T))$ and $p\in L^\frac32 (\Omega\times (0,T))$ as required in \cref{d:extendable_diss}.

To prove that $u$ has boundary extendable dissipation in the sense of \cref{d:extendable_diss}, we notice that 
\begin{equation}
    \partial_t \frac{|u^\nu|^2}{2}+\div \left(\left(\frac{|u^\nu|^2}{2}+p^\nu \right) u^\nu \right) - \nu\Delta  \frac{\abs{u^\nu}^2}{2}  = - \nu\abs{\nabla u^\nu}^2 \ \ \ (x,t) \in \overline\Omega\times (0,T].
\end{equation}
since the solutions are smooth.
Thus, for any test function $\phi \in C^\infty_c((0,T] \times \overline{\Omega})$, and since $u^\nu=0$ on $\partial \Omega$, we have that 
\begin{align}
    \int_0^T \int_{\Omega} & \frac{|u^\nu|^2}{2} \partial_t \phi \, \rmd x\,\rmd t+ \int_0^T \int_{\Omega} \left[ \left(\frac{|u^\nu|^2}{2}+p^\nu \right) u^\nu \cdot \nabla \phi +\nu \frac{\abs{u^\nu}^2}{2} \Delta\phi \right] \, \rmd x\, \rmd t \nonumber
    \\ & = \int_0^T \int_\Omega  \nu\abs{\nabla u^\nu}^2 \phi \, \rmd x\, \rmd t+ \int_{\Omega} \frac{\abs{u^\nu(x,T)}^2}{2} \phi(x,T) \, \rmd x,\label{energy balance NS smooth}
\end{align}
after integrating by parts.
Letting $\nu\rightarrow 0^+$ in \eqref{energy balance NS smooth}, by the $L^3(\Omega\times (0,T))$ strong convergence of $u^\nu$, we get
\begin{align*}
    \int_0^T \int_{\Omega} \left[ \frac{|u^\nu|^2}{2} \partial_t \phi + \frac{|u^\nu|^2}{2} u^\nu \cdot \nabla \phi + \nu \frac{\abs{u^\nu}^2}{2} \Delta\phi \right] \, \rmd x\, \rmd t\to  \int_0^T \int_{\Omega} \left[\frac{|u|^2}{2} \partial_t \phi \, \rmd x\,\rmd t+ \frac{|u|^2}{2} u \cdot \nabla \phi \right] \, \rmd x\, \rmd t.
\end{align*}
We analyze the term with the pressure. By $(i)$ we have that $u^\nu p^\nu \rightharpoonup u p$ weakly in $L^1(\Omega\times (0,T))$. Thus, we get that
\begin{equation}
    \int_0^T \int_{\Omega}p^\nu u^\nu \cdot \nabla \phi \, \rmd x\, \rmd t\to \int_0^T \int_\Omega p u \cdot \nabla \phi \, \rmd x\, \rmd t. \nonumber
\end{equation}
Finally, by $(ii)$, the right hand side of \eqref{energy balance NS smooth} converges toward 
$$ \int_0^T \int_{\overline \Omega} \phi \, \rmd \mu (x,t) + \frac12\int_{\overline \Omega} \phi(x,T) \, \rmd \lambda(x) = \int_0^T\int_{\overline \Omega} \phi \, \rmd \overline\e[u] (x,t), $$
where we set 
$$\overline\e[u] = \mu + \frac{\lambda}{2} \otimes \delta_T. $$
Since $\mu$ and $\lambda$ are both positive measures, so is $\overline\e[u]$, thereby completing the proof.
\end{proof}

\begin{remark}
    \label{r:unif_bound} 
If we assume further the uniform bound
    \begin{equation*}\label{unif_bound_p_u}
   \sup_{\nu>0}\left( \|u^\nu\|_{L^q(0,T; L^{r}(\Omega))}+\|p^\nu\|_{L^{\frac{q}{2}}(0,T; L^{\frac{r}{2}}(\Omega))}\right)<\infty,
    \end{equation*}
    for some $q,r\geq 3$, then we have, in addition, $u\in L^q(0,T; L^{r}(\Omega))$ and $p\in L^{\frac{q}{2}}(0,T; L^{\frac{r}{2}}(\Omega))$. This assumption is used in \cref{t:euler_general} and \cref{p_euler_balls}.
\end{remark}

\section{External forces}\label{s:ext_forc}
Here we give the details on how to add external forces to our analysis. Since a force $f$ affects the entropy balance \eqref{entropy_equality} for general conservation laws in a way depending on the shape of the functions $\eta$ and $Q$, we only give the explicit computations for incompressible Euler and Navier--Stokes. 

We start with the following elementary observations. If a force $f$ appears on the right hand side of the momentum equations in \eqref{NS} and \eqref{E}, then the additional term $f\cdot u$ appears in the corresponding local energy balances \eqref{nsebal} and \eqref{limitmeasure}, respectively. In particular, we need $f\cdot u \in L^1_{{\rm loc}}(\Omega\times (0,T))$, so that $f \cdot u$ is a well defined distribution. Once this is satisfied, to run the proofs of \cref{p:leray-hopf_balls} and \cref{p_euler_balls}  (while keeping the same conclusions) we must have the right quantitative estimates of $f\cdot u$ on cylinders $\mathcal{C}^\alpha_{\delta}(x,t)$. Note that, keeping the notation of the cut off functions $\chi_\delta$ and $\eta_\delta$ from \cref{s:e_proofs}, by H\"older's inequality we get
\begin{equation}
    \label{f_bound_cylind}
    \int_0^T\int_{\Omega} f\cdot u \chi_\delta \eta_\delta \, \rmd y\, \rmd \tau\lesssim \|f\cdot u\|_{L^mL^l(\mathcal{C}^\alpha_{2\delta}(x,t))} \delta^{d\frac{l-1}{l}+\alpha\frac{m-1}{m}},
\end{equation}
whenever $f\cdot u\in L^m_{{\rm loc}}(0,T;L^l_{{\rm loc}}(\Omega))$, for some $m,l\geq 1$. 

\emph{Incompressible Euler:} 
Let $s$ be as in \cref{p_euler_balls}. Thus, as soon as $m$ and $l$ satisfy 
\begin{equation}
    \label{ml_cond_e}
    d\frac{l-1}{l}+\alpha \frac{m-1}{m}\geq s,
\end{equation}
the bound \eqref{f_bound_cylind} incorporates in the estimates from the proof of \cref{p_euler_balls}, whence the same conclusions follow.  Since $u\in L^q_{{\rm loc}}(0,T;L^r_{{\rm loc}}(\Omega))$ by assumption, the interested reader can find the minimal assumption to put on $f$ such that $f\cdot u\in L^m_{{\rm loc}}(0,T;L^l_{{\rm loc}}(\Omega))$, for some $m,l$ large enough so that \eqref{ml_cond_e} is satisfied. 

The argument for boundary extendable dissipation $\overline \eps [u]$ follows in the same way. It is enough to strengthen the local integrability condition with the global one $f\cdot u\in L^m(0,T;L^l(\Omega))$.

\emph{Incompressible Navier--Stokes:} 
Let $s$ be as in \cref{p:leray-hopf_balls}. Then the  conclusion follows if 
\begin{equation}
    \label{ml_cond_ns}
    d\frac{l-1}{l}+\alpha\frac{m-1}{m}\geq s.
\end{equation}
Also in this case, we leave to the interested reader the precise computations in order to get $f\cdot u\in L^m_{{\rm loc}}(0,T;L^l_{{\rm loc}}(\Omega)) $, for some $m,l$ enjoying \eqref{ml_cond_ns}.

\section{Pressure on bounded domains for Navier--Stokes}\label{s:pressure_leray}

Here we prove a local Calderon--Zygmund type estimate on the pressure for Leray--Hopf weak solutions to the three dimensional incompressible Navier--Stokes equation \eqref{NS} on a bounded domain $\Omega$. 
\begin{proposition}
    \label{p:pressure_bound_dom}
    Let $\Omega\subset \R^3$ be a bounded domain and let $u\in L^\infty(0,T;L^2(\Omega))\cap L^2(0,T;H^1(\Omega))$ be a Leray--Hopf weak solution of \eqref{NS}. Let $r\in (2,\infty)$, $q\in [2,\infty)$ and assume that $u\in L_{\rm loc}^q(0,T;L_{\rm loc}^r(\Omega))$.
    Then the unique zero average pressure enjoys $p\in L_{{\rm loc}}^\frac{q}{2}(0,T;L^\frac{r}{2}_{{\rm loc}}(\Omega))$.
\end{proposition}
\begin{proof}
Recall the following estimate (global on $\Omega$) valid for three dimensional Leray--Hopf solutions 
    $$
\|p\|_{L^q(\eps,T;L^r(\Omega))}\lesssim \| u\cdot \nabla u\|_{L^q(\eps,T;L^r(\Omega))}<\infty,
$$
whenever $\frac{2}{q}+\frac{3}{r}=3$, $r>1$ and $\eps>0$ (see, for instance, Section 5.4 in \cite{RR}). This gives
\begin{equation}
    \label{press:est1}
    p\in L^q(\eps,T;L^1(\Omega)) \quad \forall q<\infty,
\end{equation}
since $\Omega$ is a bounded domain. Fix an open set $O\subset\joinrel\subset\Omega$ and an open interval $I\subset\joinrel\subset (\eps,T)$. Now, let $\tilde u$ be the trivial extension of $u\big|_{O}$ to the whole space $\R^3$ and let $\tilde p$ be the unique solution decaying at infinity of
 $$
 -\Delta \tilde p=\div \div (\tilde u\otimes \tilde u) \quad \text{in } \R^3.
 $$
By the standard Calderon--Zygmund estimates on the entire space we get 
\begin{equation}
    \label{pressure_est2}
    \|\tilde p\|_{ L^\frac{q}{2}(I;L^\frac{r}{2}(\R^3))}\lesssim \|\tilde u\|^2_{ L^{q}(I;L^{r}(\R^3))}=\|u\|^2_{ L^{q}(I;L^{r}(O))}.
\end{equation}
Furthermore, $p-\tilde p$ is harmonic in $O$, yielding to 
\begin{equation}
    \label{press_est3}
    \|(p-\tilde p)(t)\|_{L^\infty(\tilde \Omega)}\lesssim  \|(p-\tilde p)(t)\|_{L^1(\Omega)}, 
\end{equation}
for all $\tilde \Omega\subset\joinrel\subset O $, where the implicit constant might depend on $\tilde \Omega$. Thus, writing $p=p-\tilde p +\tilde p$ and using \eqref{press_est3}, for any $\tilde \Omega\subset\joinrel\subset O$ we estimate
\begin{align*}
\|p\|_{L^\frac{q}{2}(I;L^\frac{r}{2}(\tilde \Omega))}&\lesssim \|p-\tilde p\|_{L^\frac{q}{2}(I;L^1(\Omega))}+\|\tilde p\|_{L^\frac{q}{2}(I;L^\frac{r}{2}(\R^3))}\\
&\lesssim \| p\|_{L^\frac{q}{2}(\eps,T;L^1(\Omega))}+\|\tilde p\|_{L^\frac{q}{2}(I;L^\frac{r}{2}(\R^3))}.
\end{align*}
Then, by \eqref{press:est1} and \eqref{pressure_est2}, we infer that $p\in L_{{\rm loc}}^\frac{q}{2}(0,T;L^\frac{r}{2}_{{\rm loc}}(\Omega))$, since $O, \tilde \Omega\subset\joinrel\subset \Omega$ and $I\subset\joinrel\subset(\eps,T)$ are arbitrary (as well as
$\eps>0$). This concludes the proof.
\end{proof}

 \subsection*{Acknowledgments} LDR and MI have been partially funded by the SNF grant FLUTURA: Fluids, Turbulence, Advection No. 212573. The research of TDD was partially supported by the NSF
DMS-2106233 grant and  NSF CAREER award \#2235395. The authors would like to thank Massimo Sorella for pointing out a mistake in an early version of this manuscript.

\subsection*{Data Availability \& Conflict of Interest Statements} Data sharing not applicable to this article as no datasets were generated or analysed during the current study.  On behalf of all authors, the corresponding author states that there is no conflict of interest.


\bibliographystyle{plain}
\bibliography{biblio}

\begin{bibdiv}
\begin{biblist}

\bib{brue2022}{article}{
      author={Bru\`e, E.},
      author={De~Lellis, C.},
       title={Anomalous dissipation for the forced 3{D} {N}avier-{S}tokes
  equations},
        date={2023},
        ISSN={0010-3616,1432-0916},
     journal={Comm. Math. Phys.},
      volume={400},
      number={3},
       pages={1507\ndash 1533},
         url={https://doi.org/10.1007/s00220-022-04626-0},
      review={\MR{4595604}},
}

\bib{brue2022onsager}{article}{
      author={Brué, E.},
      author={Colombo, M.},
      author={Crippa, G.},
      author={De~Lellis, C.},
      author={Sorella, M.},
       title={Onsager critical solutions of the forced {N}avier-{S}tokes
  equations},
        date={2023},
        ISSN={1534-0392},
     journal={Comm. Pure App. Anal.},
         url={https://www.aimsciences.org/article/id/646de0334a9fed1ce4f9582d},
}

\bib{BDSV}{article}{
      author={Buckmaster, T.},
      author={Drivas, T.D.},
      author={Shkoller, S.},
      author={Vicol, V.},
       title={Simultaneous development of shocks and cusps for {2D} {E}uler
  with azimuthal symmetry from smooth data},
        date={2022},
        ISSN={2524-5317},
     journal={Ann. PDE},
      volume={8},
      number={2},
       pages={Paper No. 26},
}

\bib{CKN}{article}{
      author={Caffarelli, L.},
      author={Kohn, R.},
      author={Nirenberg, L.},
       title={Partial regularity of suitable weak solutions of the
  {N}avier-{S}tokes equations},
        date={1982},
        ISSN={0010-3640},
     journal={Comm. Pure Appl. Math.},
      volume={35},
      number={6},
       pages={771\ndash 831},
         url={https://doi.org/10.1002/cpa.3160350604},
      review={\MR{673830}},
}

\bib{C22}{article}{
      author={Cheminet, A.},
      author={Geneste, D.},
      author={Barlet, A.},
      author={Ostovan, Y.},
      author={Chaabo, T.},
      author={Valori, V.},
      author={Debue, P.},
      author={Cuvier, C.},
      author={Daviaud, F.},
      author={Foucaut, J-M.},
      author={Laval, J-P.},
      author={Padilla, V.},
      author={Wiertel-Gasquet, C.},
      author={Dubrulle, B.},
       title={Eulerian vs {L}agrangian {I}rreversibility in an experimental
  turbulent swirling flow},
        date={2022},
     journal={Physical Review Letters},
      volume={129},
      number={12},
}

\bib{CT05}{article}{
      author={Chen, G.-Q.},
      author={Torres, M.},
       title={Divergence-measure fields, sets of finite perimeter, and
  conservation laws},
        date={2005},
        ISSN={0003-9527},
     journal={Arch. Ration. Mech. Anal.},
      volume={175},
      number={2},
       pages={245\ndash 267},
         url={https://doi.org/10.1007/s00205-004-0346-1},
      review={\MR{2118477}},
}

\bib{CT11}{article}{
      author={Chen, G.-Q.},
      author={Torres, M.},
       title={On the structure of solutions of nonlinear hyperbolic systems of
  conservation laws},
        date={2011},
        ISSN={1534-0392},
     journal={Commun. Pure Appl. Anal.},
      volume={10},
      number={4},
       pages={1011\ndash 1036},
         url={https://doi.org/10.3934/cpaa.2011.10.1011},
      review={\MR{2787432}},
}

\bib{CTZ07}{article}{
      author={Chen, G.-Q.},
      author={Torres, M.},
      author={Ziemer, W.~P.},
       title={Measure-theoretic analysis and nonlinear conservation laws},
        date={2007},
        ISSN={1558-8599},
     journal={Pure Appl. Math. Q.},
      volume={3},
      number={3, Special Issue: In honor of Leon Simon. Part 2},
       pages={841\ndash 879},
         url={https://doi.org/10.4310/PAMQ.2007.v3.n3.a9},
      review={\MR{2351648}},
}

\bib{CTZ09}{article}{
      author={Chen, G.-Q.},
      author={Torres, M.},
      author={Ziemer, W.~P.},
       title={Gauss-{G}reen theorem for weakly differentiable vector fields,
  sets of finite perimeter, and balance laws},
        date={2009},
        ISSN={0010-3640},
     journal={Comm. Pure Appl. Math.},
      volume={62},
      number={2},
       pages={242\ndash 304},
         url={https://doi.org/10.1002/cpa.20262},
      review={\MR{2468610}},
}

\bib{CCFS08}{article}{
      author={Cheskidov, A.},
      author={Constantin, P.},
      author={Friedlander, S.},
      author={Shvydkoy, R.},
       title={Energy conservation and {O}nsager's conjecture for the {E}uler
  equations},
        date={2008},
        ISSN={0951-7715},
     journal={Nonlinearity},
      volume={21},
      number={6},
       pages={1233\ndash 1252},
         url={https://doi.org/10.1088/0951-7715/21/6/005},
      review={\MR{2422377}},
}

\bib{CS1}{article}{
      author={Cheskidov, A.},
      author={Shvydkoy, R.},
       title={Euler equations and turbulence: analytical approach to
  intermittency},
        date={2014},
        ISSN={0036-1410},
     journal={SIAM J. Math. Anal.},
      volume={46},
      number={1},
       pages={353\ndash 374},
         url={https://doi.org/10.1137/120876447},
      review={\MR{3152734}},
}

\bib{CS2}{article}{
      author={Cheskidov, A.},
      author={Shvydkoy, R.},
       title={Volumetric theory of intermittency in fully developed
  turbulence},
        date={2023},
        ISSN={0003-9527,1432-0673},
     journal={Arch. Ration. Mech. Anal.},
      volume={247},
      number={3},
       pages={Paper No. 45, 35},
         url={https://doi.org/10.1007/s00205-023-01878-5},
      review={\MR{4581460}},
}

\bib{crippa}{article}{
      author={Colombo, M.},
      author={Crippa, G.},
      author={Sorella, M.},
       title={Anomalous dissipation and lack of selection in the
  {O}bukhov-{C}orrsin theory of scalar turbulence},
        date={2023},
        ISSN={2524-5317,2199-2576},
     journal={Ann. PDE},
      volume={9},
      number={2},
       pages={Paper No. 21, 48},
         url={https://doi.org/10.1007/s40818-023-00162-9},
      review={\MR{4662772}},
}

\bib{DK22}{article}{
      author={De~Lellis, C.},
      author={Kwon, H.},
       title={On nonuniqueness of {H}\"{o}lder continuous globally dissipative
  {E}uler flows},
        date={2022},
        ISSN={2157-5045,1948-206X},
     journal={Anal. PDE},
      volume={15},
      number={8},
       pages={2003\ndash 2059},
         url={https://doi.org/10.2140/apde.2022.15.2003},
      review={\MR{4546502}},
}

\bib{DRH23}{article}{
      author={De~Rosa, L.},
      author={Haffter, S.},
       title={A fractal version of the {O}nsager's conjecture: the
  {$\beta-$}model},
        date={2023},
        ISSN={0002-9939},
     journal={Proc. Amer. Math. Soc.},
      volume={151},
      number={1},
       pages={255\ndash 267},
         url={https://doi.org/10.1090/proc/16104},
      review={\MR{4504623}},
}

\bib{DRI22}{article}{
      author={De~Rosa, L.},
      author={Isett, P.},
       title={Intermittency and lower dimensional dissipation in incompressible
  fluids},
        date={2024},
        ISSN={0003-9527,1432-0673},
     journal={Arch. Ration. Mech. Anal.},
      volume={248},
      number={1},
       pages={Paper No. 11, 37},
         url={https://doi.org/10.1007/s00205-023-01954-w},
      review={\MR{4694411}},
}

\bib{D19}{article}{
      author={Drivas, T.~D.},
       title={Turbulent cascade direction and {L}agrangian time-asymmetry},
        date={2019},
        ISSN={0938-8974},
     journal={J. Nonlinear Sci.},
      volume={29},
      number={1},
       pages={65\ndash 88},
         url={https://doi.org/10.1007/s00332-018-9476-8},
      review={\MR{3908861}},
}

\bib{drivas2022anomalous}{article}{
      author={Drivas, T.~D.},
      author={Elgindi, T.~M.},
      author={Iyer, G.},
      author={Jeong, I-J.},
       title={Anomalous dissipation in passive scalar transport},
        date={2022},
        ISSN={0003-9527},
     journal={Arch. Ration. Mech. Anal.},
      volume={243},
      number={3},
       pages={1151\ndash 1180},
         url={https://doi.org/10.1007/s00205-021-01736-2},
      review={\MR{4381138}},
}

\bib{DE18}{article}{
      author={Drivas, T.~D.},
      author={Eyink, G.~L.},
       title={An {O}nsager singularity theorem for turbulent solutions of
  compressible {E}uler equations},
        date={2018},
        ISSN={0010-3616},
     journal={Comm. Math. Phys.},
      volume={359},
      number={2},
       pages={733\ndash 763},
         url={https://doi.org/10.1007/s00220-017-3078-4},
      review={\MR{3783560}},
}

\bib{DE19}{article}{
      author={Drivas, T.~D.},
      author={Eyink, G.~L.},
       title={An {O}nsager singularity theorem for {L}eray solutions of
  incompressible {N}avier-{S}tokes},
        date={2019},
        ISSN={0951-7715},
     journal={Nonlinearity},
      volume={32},
      number={11},
       pages={4465\ndash 4482},
         url={https://doi.org/10.1088/1361-6544/ab2f42},
      review={\MR{4017111}},
}

\bib{DN18}{article}{
      author={Drivas, T.~D.},
      author={Nguyen, H.~Q.},
       title={Onsager's conjecture and anomalous dissipation on domains with
  boundary},
        date={2018},
        ISSN={0036-1410},
     journal={SIAM J. Math. Anal.},
      volume={50},
      number={5},
       pages={4785\ndash 4811},
         url={https://doi.org/10.1137/18M1178864},
      review={\MR{3851045}},
}

\bib{DN19}{article}{
      author={Drivas, T.~D.},
      author={Nguyen, H.~Q.},
       title={Remarks on the emergence of weak {E}uler solutions in the
  vanishing viscosity limit},
        date={2019},
        ISSN={0938-8974},
     journal={J. Nonlinear Sci.},
      volume={29},
      number={2},
       pages={709\ndash 721},
         url={https://doi.org/10.1007/s00332-018-9500-z},
      review={\MR{3927111}},
}

\bib{DG22}{article}{
      author={Dubrulle, B.},
      author={Gibbon, J.~D.},
       title={A correspondence between the multifractal model of turbulence and
  the {N}avier-{S}tokes equations},
        date={2022},
        ISSN={1364-503X},
     journal={Philos. Trans. Roy. Soc. A},
      volume={380},
      number={2218},
       pages={Paper No. 20210092, 10},
      review={\MR{4395524}},
}

\bib{DR00}{article}{
      author={Duchon, J.},
      author={Robert, R.},
       title={Inertial energy dissipation for weak solutions of incompressible
  {E}uler and {N}avier-{S}tokes equations},
        date={2000},
        ISSN={0951-7715},
     journal={Nonlinearity},
      volume={13},
      number={1},
       pages={249\ndash 255},
         url={https://doi.org/10.1088/0951-7715/13/1/312},
      review={\MR{1734632}},
}

\bib{EG15}{book}{
      author={Evans, L.~C.},
      author={Gariepy, R.~F.},
       title={Measure theory and fine properties of functions},
     edition={Revised},
      series={Textbooks in Mathematics},
   publisher={CRC Press, Boca Raton, FL},
        date={2015},
        ISBN={978-1-4822-4238-6},
      review={\MR{3409135}},
}

\bib{Enotes}{article}{
      author={Eyink, G.~L.},
       title={Turbulence theory},
        date={2015},
        note={Course notes available at
  \href{https://www.ams.jhu.edu/~eyink/Turbulence/}{Eyink/Turbulence}},
}

\bib{Eyink95}{incollection}{
      author={Eyink, G.~L.},
       title={Besov spaces and the multifractal hypothesis},
        date={1995},
      volume={78},
       pages={353\ndash 375},
         url={https://doi.org/10.1007/BF02183353},
        note={Papers dedicated to the memory of Lars Onsager},
      review={\MR{1317149}},
}

\bib{Falc}{book}{
      author={Falconer, K.},
       title={Fractal geometry},
     edition={Third},
   publisher={John Wiley \& Sons, Ltd., Chichester},
        date={2014},
        ISBN={978-1-119-94239-9},
        note={Mathematical foundations and applications},
      review={\MR{3236784}},
}

\bib{Frisch91}{incollection}{
      author={Frisch, U.},
       title={From global scaling, \`a la {K}olmogorov, to local multifractal
  scaling in fully developed turbulence},
        date={1991},
      volume={434},
       pages={89\ndash 99},
         url={https://doi.org/10.1098/rspa.1991.0082},
        note={Turbulence and stochastic processes: Kolmogorov's ideas 50 years
  on},
      review={\MR{1124929}},
}

\bib{Frisch95}{book}{
      author={Frisch, U.},
       title={Turbulence},
   publisher={Cambridge University Press, Cambridge},
        date={1995},
        ISBN={0-521-45103-5},
        note={The legacy of A. N. Kolmogorov},
      review={\MR{1428905}},
}

\bib{Gibbon22}{article}{
      author={Gibbon, J.~D.},
       title={Identifying the multifractal set on which energy dissipates in a
  turbulent {N}avier-{S}tokes fluid},
        date={2023},
        ISSN={0167-2789,1872-8022},
     journal={Phys. D},
      volume={445},
       pages={Paper No. 133654, 4},
         url={https://doi.org/10.1016/j.physd.2023.133654},
      review={\MR{4533965}},
}

\bib{GKN23}{article}{
      author={Giri, V.},
      author={Kwon, H.},
      author={Novack, M.},
       title={The $\text{L}^3$-based strong {O}nsager theorem},
        date={2023},
        note={Preprint available at
  \href{https://arxiv.org/abs/2305.18509}{arXiv:2305.18509}},
}

\bib{Is17}{article}{
      author={Isett, P.},
       title={On the endpoint regularity in {O}nsager's conjecture},
        date={2017},
        note={Preprint available at
  \href{https://arxiv.org/abs/1706.01549}{arXiv:1706.01549}},
}

\bib{I_local22}{article}{
      author={Isett, P.},
       title={Nonuniqueness and existence of continuous, globally dissipative
  {E}uler flows},
        date={2022},
        ISSN={0003-9527,1432-0673},
     journal={Arch. Ration. Mech. Anal.},
      volume={244},
      number={3},
       pages={1223\ndash 1309},
         url={https://doi.org/10.1007/s00205-022-01780-6},
      review={\MR{4419613}},
}

\bib{ISS03}{article}{
      author={Iskauriaza, L.},
      author={Ser\"{e}gin, G.~A.},
      author={Šverák, V.},
       title={{$L_{3,\infty}$}-solutions of {N}avier--{S}tokes equations and
  backward uniqueness},
    language={Russian, with Russian summary},
        date={2003},
        ISSN={0042-1316},
     journal={Uspekhi Mat. Nauk},
      volume={58},
      number={2(350)},
       pages={3\ndash 44},
}

\bib{ISY20}{article}{
      author={Iyer, K.~P.},
      author={Sreenivasan, K.~R.},
      author={Yeung, P.~K.},
       title={Scaling exponents saturate in three-dimensional isotropic
  turbulence},
        date={2020},
     journal={Physical Review Fluids},
      volume={5.5},
}

\bib{JY21}{article}{
      author={Jeong, I-J.},
      author={Yoneda, T.},
       title={Vortex stretching and enhanced dissipation for the incompressible
  3{D} {N}avier-{S}tokes equations},
        date={2021},
        ISSN={0025-5831},
     journal={Math. Ann.},
      volume={380},
      number={3-4},
       pages={2041\ndash 2072},
         url={https://doi.org/10.1007/s00208-020-02019-z},
      review={\MR{4297205}},
}

\bib{KIYIU03}{article}{
      author={Kaneda, Y.},
      author={Ishihara, T.},
      author={Yokokawa, M.},
      author={Itakura, K.},
      author={Uno, A.},
       title={Energy dissipation rate and energy spectrum in high resolution
  direct numerical simulations of turbulence in a periodic box},
        date={2003},
     journal={Phys. Fluids},
      volume={15},
}

\bib{K41}{article}{
      author={Kolmogoroff, A.~N.},
       title={Dissipation of energy in the locally isotropic turbulence},
        date={1941},
     journal={C. R. (Doklady) Acad. Sci. URSS (N.S.)},
      volume={32},
       pages={16\ndash 18},
      review={\MR{0005851}},
}

\bib{K62}{article}{
      author={Kolmogorov, A.~N.},
       title={A refinement of previous hypotheses concerning the local
  structure of turbulence in a viscous incompressible fluid at high {R}eynolds
  number},
        date={1962},
        ISSN={0022-1120},
     journal={J. Fluid Mech.},
      volume={13},
       pages={82\ndash 85},
         url={https://doi.org/10.1017/S0022112062000518},
      review={\MR{139329}},
}

\bib{LL}{book}{
      author={Landau, L.D.},
      author={Lifshitz, E.~M.},
       title={Fluid {M}echanics: {L}andau and {L}ifshitz: {C}ourse of
  {T}heoretical {P}hysics},
   publisher={Elsevier},
        date={2013},
      volume={6},
}

\bib{L34}{article}{
      author={Leray, J.},
       title={Sur le mouvement d'un liquide visqueux emplissant l'espace},
        date={1934},
        ISSN={0001-5962},
     journal={Acta Math.},
      volume={63},
      number={1},
       pages={193\ndash 248},
         url={https://doi.org/10.1007/BF02547354},
      review={\MR{1555394}},
}

\bib{LS}{article}{
      author={Leslie, T.~M.},
      author={Shvydkoy, R.},
       title={Conditions implying energy equality for weak solutions of the
  {N}avier-{S}tokes equations},
        date={2018},
        ISSN={0036-1410},
     journal={SIAM J. Math. Anal.},
      volume={50},
      number={1},
       pages={870\ndash 890},
         url={https://doi.org/10.1137/16M1104147},
      review={\MR{3759872}},
}

\bib{LS2}{article}{
      author={Leslie, T.~M.},
      author={Shvydkoy, R.},
       title={The energy measure for the {E}uler and {N}avier-{S}tokes
  equations},
        date={2018},
        ISSN={0003-9527},
     journal={Arch. Ration. Mech. Anal.},
      volume={230},
      number={2},
       pages={459\ndash 492},
         url={https://doi.org/10.1007/s00205-018-1250-4},
      review={\MR{3842051}},
}

\bib{L60}{article}{
      author={Lions, J.~L.},
       title={Sur la r\'{e}gularit\'{e} et l'unicit\'{e} des solutions
  turbulentes des \'{e}quations de {N}avier {S}tokes},
        date={1960},
        ISSN={0041-8994},
     journal={Rend. Sem. Mat. Univ. Padova},
      volume={30},
       pages={16\ndash 23},
         url={http://www.numdam.org/item?id=RSMUP_1960__30__16_0},
      review={\MR{115018}},
}

\bib{Majda}{article}{
      author={Majda, A.},
       title={The existence of multidimensional shock fronts},
        date={1983},
        ISSN={0065-9266},
     journal={Mem. Amer. Math. Soc.},
      volume={43},
      number={281},
       pages={v+93},
         url={https://doi.org/10.1090/memo/0281},
      review={\MR{699241}},
}

\bib{Mandelbrot}{article}{
      author={Mandelbrot, B.~B.},
       title={Intermittent turbulence in self-similar cascades: divergence of
  high moments and dimension of the carrier},
        date={1974},
     journal={Journal of fluid Mechanics},
      volume={62},
      number={2},
       pages={331\ndash 358},
}

\bib{MAT}{book}{
      author={Mattila, P.},
       title={Geometry of sets and measures in {E}uclidean spaces},
      series={Cambridge Studies in Advanced Mathematics},
   publisher={Cambridge University Press, Cambridge},
        date={1995},
      volume={44},
        ISBN={0-521-46576-1; 0-521-65595-1},
         url={https://doi.org/10.1017/CBO9780511623813},
        note={Fractals and rectifiability},
      review={\MR{1333890}},
}

\bib{Meneveau91}{article}{
      author={Meneveau, C.},
      author={Sreenivasan, K.~R.},
       title={The multifractal nature of turbulent energy dissipation},
        date={1991},
     journal={Journal of Fluid Mechanics},
      volume={224},
       pages={429\ndash 484},
}

\bib{Meneveau87}{article}{
      author={Meneveau, C.},
      author={Sreenivasan, K.~R.},
       title={The multifractal spectrum of the dissipation field in turbulent
  flows},
        date={1987},
     journal={Nuclear Physics B-Proceedings Supplements},
      volume={2},
       pages={49\ndash 76},
}

\bib{Meneveau88}{article}{
      author={Meneveau, C.},
      author={Sreenivasan, K.~R.},
       title={Singularities of the equations of fluid motion},
        date={1988},
     journal={Physical Review A},
      volume={38},
      number={12},
}

\bib{Farge}{article}{
      author={Nguyen~van yen, N.},
      author={Waidmann, M.},
      author={Klein, R.},
      author={Farge, M.},
      author={Schneider, K.},
       title={Energy dissipation caused by boundary layer instability at
  vanishing viscosity},
        date={2018},
        ISSN={0022-1120},
     journal={J. Fluid Mech.},
      volume={849},
       pages={676\ndash 717},
         url={https://doi.org/10.1017/jfm.2018.396},
      review={\MR{3819084}},
}

\bib{O49}{article}{
      author={Onsager, L.},
       title={Statistical hydrodynamics},
        date={1949},
     journal={Nuovo Cimento (9)},
      volume={6},
      number={Supplemento, 2 (Convegno Internazionale di Meccanica
  Statistica)},
       pages={279\ndash 287},
}

\bib{PKW02}{article}{
      author={Pearson, B.~R.},
      author={Krongstad, P.~A.},
      author={van~de Water, W.},
       title={Measurements of the turbulent energy dissipation rate},
        date={2002},
     journal={Phys. Fluids},
      volume={14},
       pages={1288\ndash 1290},
}

\bib{PT08}{article}{
      author={Phuc, N.~C.},
      author={Torres, M.},
       title={Characterizations of the existence and removable singularities of
  divergence-measure vector fields},
        date={2008},
        ISSN={0022-2518},
     journal={Indiana Univ. Math. J.},
      volume={57},
      number={4},
       pages={1573\ndash 1597},
         url={https://doi.org/10.1512/iumj.2008.57.3312},
      review={\MR{2440874}},
}

\bib{P59}{article}{
      author={Prodi, G.},
       title={Un teorema di unicit\`a per le equazioni di {N}avier-{S}tokes},
        date={1959},
        ISSN={0003-4622},
     journal={Ann. Mat. Pura Appl. (4)},
      volume={48},
       pages={173\ndash 182},
         url={https://doi.org/10.1007/BF02410664},
      review={\MR{126088}},
}

\bib{RR}{book}{
      author={Robinson, J.~C.},
      author={Rodrigo, J.~L.},
      author={Sadowski, W.},
       title={The three-dimensional {N}avier-{S}tokes equations},
      series={Cambridge Studies in Advanced Mathematics},
   publisher={Cambridge University Press, Cambridge},
        date={2016},
      volume={157},
        ISBN={978-1-107-01966-9},
         url={https://doi.org/10.1017/CBO9781139095143},
        note={Classical theory},
      review={\MR{3616490}},
}

\bib{S62}{article}{
      author={Serrin, J.},
       title={On the interior regularity of weak solutions of the
  {N}avier-{S}tokes equations},
        date={1962},
        ISSN={0003-9527},
     journal={Arch. Rational Mech. Anal.},
      volume={9},
       pages={187\ndash 195},
         url={https://doi.org/10.1007/BF00253344},
      review={\MR{136885}},
}

\bib{S74}{article}{
      author={Shinbrot, M.},
       title={The energy equation for the {N}avier-{S}tokes system},
        date={1974},
        ISSN={0036-1410},
     journal={SIAM J. Math. Anal.},
      volume={5},
       pages={948\ndash 954},
         url={https://doi.org/10.1137/0505092},
      review={\MR{435629}},
}

\bib{S83}{article}{
      author={Sohr, H.},
       title={Zur {R}egularit\"{a}tstheorie der instation\"{a}ren {G}leichungen
  von {N}avier--{S}tokes},
    language={German},
        date={1983},
     journal={Math. Z.},
      volume={184},
      number={3},
       pages={359\ndash 375},
}

\bib{KRS84}{article}{
      author={Sreenivasan, K.~R.},
       title={On the scaling of the turbulence energy dissipation rate},
        date={1984},
     journal={Phys. Fluids},
      volume={27},
       pages={1048\ndash 1051},
}

\bib{KRS98}{article}{
      author={Sreenivasan, K.~R.},
       title={An update on the energy dissipation rate in isotropic
  turbulence},
        date={1998},
        ISSN={1070-6631},
     journal={Phys. Fluids},
      volume={10},
      number={2},
       pages={528\ndash 529},
         url={https://doi.org/10.1063/1.869575},
      review={\MR{1604680}},
}

\bib{Taylor17}{article}{
      author={Taylor, G.~I.},
       title={Motion of solids in fluids when the flow is not irrotational},
        date={1917},
     journal={Proc. Royal Soc. A},
      volume={648},
       pages={99\ndash 113},
}

\bib{S05}{article}{
      author={\v{S}ilhav\'{y}, M.},
       title={Divergence measure fields and {C}auchy's stress theorem},
        date={2005},
        ISSN={0041-8994},
     journal={Rend. Sem. Mat. Univ. Padova},
      volume={113},
       pages={15\ndash 45},
      review={\MR{2168979}},
}

\bib{Wu21}{article}{
      author={Wu, B.},
       title={Partially regular weak solutions of the {N}avier-{S}tokes
  equations in {$\Bbb R^4\times[0,\infty[$}},
        date={2021},
        ISSN={0003-9527,1432-0673},
     journal={Arch. Ration. Mech. Anal.},
      volume={239},
      number={3},
       pages={1771\ndash 1808},
         url={https://doi.org/10.1007/s00205-020-01603-6},
      review={\MR{4215201}},
}

\bib{Wu22}{article}{
      author={Wu, B.},
       title={Partially regular weak solutions of the stationary
  {N}avier-{S}tokes equations in dimension 6},
        date={2022},
        ISSN={0944-2669,1432-0835},
     journal={Calc. Var. Partial Differential Equations},
      volume={61},
      number={4},
       pages={Paper No. 152, 25},
         url={https://doi.org/10.1007/s00526-022-02273-w},
      review={\MR{4438343}},
}

\end{biblist}
\end{bibdiv}

\end{document}